\documentclass{amsart}

\usepackage{cancel}
\usepackage{soul}

\usepackage{amssymb,amsmath,graphicx}
\usepackage{xcolor}
\usepackage{enumerate}
\newtoks\prt

\newtheorem{thm}{Theorem}[section]

\newtheorem{lemma}[thm]{Lemma}
\newtheorem{prop}[thm]{Proposition}
\newtheorem{cor}[thm]{Corollary}

\theoremstyle{definition}

\newtheorem{remark}[thm]{Remark}
\newtheorem*{rem}{Remark}

\def\eqn#1$$#2$${\begin{equation}\label#1#2\end{equation}}
 \def\J#1#2#3{ \left\{ #1,#2,#3 \right\} }

\def\U{\mathcal U}

\def\P{\mathcal P}

\def\ce{\mathbb C}

\def\en{\mathbb N}

\def \ran {\operatorname{ran}}

\def \ext {\operatorname{ext}}

\def \reg {\partial _{\kern1pt\text{reg}}}

\def\ip#1#2{\left\langle#1,#2\right\rangle}

\def\di{\,\mbox{\rm d}}

\newcommand{\norm}[1]{\left\|#1\right\|}

\renewcommand{\Re}{\operatorname{Re}}

\newcommand{\abs}[1]{\left|#1\right|}
\newcommand{\setsep}{;\,}

\title[Grothendieck's inequalities]{Grothendieck's inequalities for JB$^*$-triples: Proof of the Barton--Friedman conjecture}

\author[J. Hamhalter]{Jan Hamhalter}

\author[O.F.K. Kalenda]{Ond\v{r}ej F.K. Kalenda}

\author[A.M. Peralta]{Antonio M. Peralta}

\author[H. Pfitzner]{Hermann Pfitzner}

\address{Czech Technical University in Prague, Faculty of Electrical Engineering, Department of Mathematics, Technicka 2, 166 27, Prague 6,
Czech Republic}
\email{hamhalte@math.feld.cvut.cz}
\address{Charles University, Faculty of Mathematics and Physics, Department of
Mathematical Analysis, Sokolovsk{\'a} 86, 186 75 Praha 8, Czech Republic}
\email{kalenda@karlin.mff.cuni.cz}
\address{Departamento de An{\'a}lisis Matem{\'a}tico, Facultad de
Ciencias, Universidad de Gra\-na\-da, 18071 Granada, Spain.}
\email{aperalta@ugr.es}
\address{Universit\'{e} d'Orl\'{e}ans,
BP 6759,
F-45067 Orl\'{e}ans Cedex 2,
France}
\email{pfitzner@univ-orleans.fr}

\thanks{The first two authors were in part supported by the Research Grant GA\v{C}R 17-00941S. The first author was partly supported further by the project OP VVV Center for Advanced Applied Science CZ.02.1.01/0.0/0.0/16\_019/000077.  The third author was partially supported by Junta de Andaluc\'{\i}a grant FQM375.}



\begin{document}

\begin{abstract}
We prove that, given a constant $K> 2$ and a bounded linear operator $T$ from a JB$^*$-triple $E$ into a complex Hilbert space $H$, there exists a norm-one functional $\psi\in E^*$ satisfying  $$\|T(x)\| \leq K \, \|T\| \, \|x\|_{\psi},$$  for all $x\in E$. Applying this result we show that, given $G > 8 (1+2\sqrt{3})$  and a bounded bilinear form $V$ on the Cartesian product of two JB$^*$-triples $E$ and $B$, there exist norm-one functionals $\varphi\in E^{*}$ and $\psi\in B^{*}$ satisfying
$$|V(x,y)| \leq G \ \|V\| \, \|x\|_{\varphi} \, \|y\|_{\psi}$$ for all $(x,y)\in E \times B$. 
These results prove a conjecture pursued during almost twenty years. 
\end{abstract}

\maketitle

\section{Introduction}

In order to review the historical emplacement of a conjecture open for almost twenty years, we should turn back to the fifties, to a major contribution in functional analysis. Grothendieck's inequalities and Grothendieck's constants were named after A. Grothendieck, who established the first result in this direction in his celebrated ``\emph{R{\'e}sum{\'e} de la th{\'e}orie m{\'e}trique des produits tensoriels topologiques}'' (see \cite{grothendieck1956resume}). Grothendieck's original result proves the existence of a universal constant $G>0$ (called Grothendieck's constant), satisfying that for every couple $(\Omega_1,\Omega_2)$ of compact Hausdorff spaces and every bilinear form $V$ on
$C(\Omega_1)\times C(\Omega_2)$ there exist two probability measures $\mu_1$ and $\mu_2$ on $\Omega_1$ and $\Omega_2$, respectively, such that $$\left|V(f,g) \right| \leq G \|V\| \, \left(\int_{\Omega_1} |f(t)|^2 d\mu_1(t) \right)^{\frac12} \, \left(\int_{\Omega_2} |g(s)|^2 d\mu_2(s) \right)^{\frac12}$$
for all $f\in C(\Omega_1)$ and $g\in C(\Omega_2)$. In 1956, Grothendieck predicted the validity of the previous result when the space $C(\Omega),$ of all complex valued continuous functions on a compact Hausdorff space $\Omega$, is replaced with a general C$^*$-algebra (cf. \cite[\S 6, Question 4]{grothendieck1956resume}). Grothendieck's forethought was confirmed several years later. In subsequent remarkable contributions, G. Pisier \cite{pisier1978grothendieck} and U. Haagerup \cite{haagerup1985grothendieck} established the so-called \emph{non-commutative Grothendieck inequality}, which assures that for every bounded bilinear form $V$ on the cartesian product of two C$^*$-algebras $A$ and $B$, there exist two states $\phi$ in $A^*$ and $\psi\in B^*$ satisfying
$$|V(x,y)| \leq 4 \, \|V\|\, \phi \left( \frac{x x^* + x^* x}{2}\right)^{\frac12} \psi \left( \frac{y y^* + y^* y}{2}\right)^{\frac12},$$ for all $(x,y)\in A\times B$. Briefly speaking, at the cost of replacing probability measures with states and moduli of continuous functions with absolute values of the form $|x|^2 = \frac{x x^* + x^* x}{2}$ ($x\in A$), the Grothendieck's inequality works for bounded bilinear forms on the Cartesian product of two C$^*$-algebras. That is, in the non-commutative setting, the pre-Hilbertian semi-norms of the form $\|x\|_{\phi}^2:= \phi \left( \frac{x x^* + x^* x}{2}\right)$, where $\phi$ runs through the set of all states on a C$^*$-algebra $A$, are valid to factor all bounded bilinear forms.\smallskip

There exists a class of complex Banach spaces, called JB$^*$-triples, which are determined by the holomorphic properties of their open unit balls (see Subsection~\ref{S:1.1} below for details). The class of JB$^*$-triples includes (among others) all C$^*$-algebras, and all complex Hilbert spaces. We therefore have a strictly wider class of complex Banach spaces than that determined by all C$^*$-algebras. The setting of JB$^*$-triples seemed an appropriate candidate to extend the Grothendieck's inequality when in 1987 J.T. Barton and Y. Friedman explored this problem.\smallskip 

Although JB$^*$-triples lack an order structure like the one appearing in the setting of C$^*$-algebras, every JB$^*$-triple $E$ admits a large collection of pre-Hilbertian semi-norms which arise naturally from the geometric structure and play a similar role to those determined by the states on a C$^*$-algebra. Barton and Friedman showed in \cite{barton1987grothendieck} that for each norm-one functional $\varphi$ in the dual, $E^*$, of $E$, and each norm-one element $z$ in $E^{**}$ with $\varphi (z) =1$, the mapping $x\mapsto \|x\|_{\varphi}= \varphi\{x,x,z\}^{\frac12}$ defines a pre-Hilbert semi-norm on $E$ which does not depend on the choice of the element $z$. Let us observe that if $\phi$ is a state on a C$^*$-algebra $A$ and $1$ denotes the unit element in $A^{**}$, then $\phi (1)=1$ and  $\|x\|_{\phi}^2 = \phi \{x,x,1\} = \phi \left( \frac{x x^* + x^* x}{2}\right)$ for all $x\in A$. Theorem 1.4 in \cite{barton1987grothendieck} asserts the existence of a universal constant $K\in [2,3+2 \sqrt{2}]$ satisfying the following property: for every bounded bilinear form $V$ on the cartesian product of two JB$^*$-triples $E$ and $F$ there exist norm-one functionals $\varphi \in E^{*}$ and $\psi \in F^{*}$ satisfying \begin{equation}\label{eq BF Grothendieck inequ} |V(x,y)| \leq K \, \|V\| \, \|x\|_{\varphi} \, \|y\|_{\psi},
 \end{equation} for all $(x,y)\in E \times F$.
Building upon the results published in \cite{barton1987grothendieck}, Ch.-H. Chu, B. Iochum and G. Loupias gave an alternative proof of this result in \cite[Theorem 6]{chu1989grothendieck}.\smallskip 

Grothendieck's inequalities were revisited in the setting of real JB$^*$-triples at the beginning of this century, and it was pointed out in \cite{peralta2001little,peralta2001grothendieck} that the proof of \cite[Theorem 1.3]{barton1987grothendieck}
contains a gap affecting also
the arguments and conclusions in \cite{chu1989grothendieck}.
As a consequence of these difficulties, the original statement of Grothendieck's inequality for JB$^*$-triples in \eqref{eq BF Grothendieck inequ} can not be considered as proved, and it is nowadays known as the \emph{Barton--Friedman conjecture}.\smallskip

The main results in \cite{peralta2001little,peralta2001grothendieck,peralta2001grothendieckrev} show that, at the cost of replacing semi-norms of the form $\|.\|_{\varphi}$ and $\|.\|_{\psi}$ with semi-norms of the form $\|.\|_{\varphi_1,\varphi_2}$, $\|.\|_{\psi_1,\psi_2}$ for convenient norm-one functionals ${\varphi_1,\varphi_2}\in E^*$ and $\psi_1,\psi_2\in F^*$, the conclusion in \eqref{eq BF Grothendieck inequ} is true for $K>4(1+2 \sqrt{3})$ (cf. \cite[Theorem 6]{peralta2001grothendieck}). Let us remark that for ${\varphi_1,\varphi_2}\in E^*$ we set $\|x\|_{\varphi_1,\varphi_2}^2 := \|x\|_{\varphi_1}^2 + \|x\|_{\varphi_2}^2$ ($x\in E$). This result was applied to dissipate the concerns affecting subsequent results in JB$^*$-triple theory (for example, properties of the strong$^*$-topology, characterization of weakly compact operators from a JB$^*$-triple into a Banach space, etc.) whose proofs depended on the original form of Grothendieck's inequality by Barton and Friedman. Despite these advances, the Barton--Friedman conjecture (i.e. the statement in \eqref{eq BF Grothendieck inequ}) was neither proven nor discarded.\smallskip

In \cite{peralta2005new} the Barton-Friedman conjecture was proved in some special cases -- for Cartan factors and atomic JBW$^*$-triples (i.e. $\ell_{\infty}$-sums of Cartan factors).\smallskip

In 2012, G. Pisier wrote ``\emph{The problem of extending the non-commutative Gro\-then\-dieck theorem from C$^*$-algebras
to JB$^*$-triples was considered notably by Barton and Friedman around 1987, but seems to be still incomplete}'' (cf. \cite[Remark 8.3]{pisier2012grothendieck}). The recent monograph \cite{Cabrera-Rodriguez-vol2} deals with the Barton--Friedman conjecture under an equivalent reformulation in terms of the \emph{little Grothendieck inequality} (see \cite[Problem 5.10.131]{Cabrera-Rodriguez-vol2}). We refer to section \ref{sec:little} for more details on the little Grothendieck inequality. It is very well illustrated in \cite[pages 337-346]{Cabrera-Rodriguez-vol2} how a proof to the Barton-Friedman conjecture, or equivalently, to the little Grothendieck inequality, might have important consequences and \emph{``restore the validity''} of all subsequent works relying on the original Grothendieck inequality in \eqref{eq BF Grothendieck inequ}.\smallskip

In this paper we fill the gap by proving the Barton-Friedman conjecture. The main result reads as follows

\begin{thm}\label{tt big Grothendieck inequ} Suppose $G >{ 8 (1+2\sqrt{3})}$. Let $E$ and $B$ be JB$^*$-triples. Then for every bounded bilinear form $V:E\times B\to \mathbb{C}$ there exist norm-one functionals $\varphi\in E^{*}$ and $\psi\in B^{*}$ satisfying
$$|V(x,y)| \leq G \ \|V\| \, \|x\|_{\varphi} \, \|y\|_{\psi}$$ for all $(x,y)\in E \times B$.$\hfill\Box$
\end{thm}

This theorem will be proved in Theorem \ref{t big Grothendieck inequ} below. 

The paper is organized as follows. In Subsection~\ref{S:1.1} we provide some background on JB$^*$-triples. Subsection~\ref{s:represent} deals with a representation of JBW$^*$-triples in the form of a suitable direct sum (see Proposition~\ref{P:representation}). 

Section~\ref{sec:little} is devoted to the so-called little Grothendieck inequality. We recall where the gap was and indicate the strategy of our proof.

In the three following sections we prove the little Grothendieck inequality for individual summands from Proposition~\ref{P:representation}
and in the last section we glue the results together and provide proofs of the main results.

Along the paper, all Banach spaces will be over the field of complex numbers, the symbols $S_{X}$ and  $B_{X}$ will stand for the unit sphere and the closed unit ball of  a Banach space $X$, respectively.

\subsection{Basic notions and nomenclature}\label{S:1.1}

The aim of extending the celebrated Riemann mapping theorem to complex Banach spaces of arbitrary dimension led W. Kaup to classify bounded symmetric domains in arbitrary complex Banach spaces (see \cite{kaup1983riemann}). It was proved by L. Harris that the open unit ball of a C$^*$-algebra is a bounded symmetric domain (cf. \cite{harris1974bounded}). It should be recalled that a domain $\mathcal{D}$ in a complex Banach space is symmetric if for each $a$ in $\mathcal{D}$ there is a biholomorphic map $S_a$ of $\mathcal{D}$ onto itself with $S_a=  S_a^{-1}$, such that $a$ is an isolated fixed point of $S_a$. However, the open unit balls of all C$^*$-algebras do not exhaust all examples, namely, infinite dimensional complex Hilbert spaces enjoy the same property, but they are never C$^*$-algebras. The celebrated contribution due to W. Kaup shows that the biholomorphic images of the open unit balls of JB$^*$-triples cover all possible examples of bounded symmetric domains (cf. \cite{kaup1983riemann} or \cite[Theorem 2.5.27]{chubook}).\smallskip

A  \emph{JB$^*$-triple} is a complex Banach space $E$  equipped with a (continuous) triple product $\{.,.,.\}: E^3 \to E$, which is symmetric and bilinear in the outer variables and conjugate-linear in the middle one, and satisfies the following algebraic--analytic axioms
(where given $a,b\in E$, $L(a,b)$ stands for the (linear) operator on $E$ given by $L(a,b)(x)=\J abx$, for all $x\in E$):
\begin{enumerate}[$(JB^*$-$1)$]
\item $L(x,y)L(a,b) = L(L(x,y)(a),b) $ $- L(a,L(y,x)(b))$ $+ L(a,b) L(x,y),$ for all $a,b,x,y\in E$; \hfill \emph{(Jordan identity)}
\item The operator $L(a,a)$ is a hermitian operator with nonnegative spectrum for each $a\in E$;
\item $\|\{a,a,a\}\|=\|a\|^3$ for $a\in E$.
\end{enumerate}

The space $B(H,K)$ of all bounded linear operators between complex Hilbert spaces $H$ and $K$, which is rarely a C$^*$-algebra, is always a JB$^*$-triple when equipped with the triple product defined by $\J xyz =\frac12 (x y^* z +z y^* x).$ The same triple product provides a structure of JB$^*$-triple for every C$^*$-algebra. Moreover, if $H$ is a complex Hilbert space, it can be canonically identified with $B(\ce,H)$, so the above triple triple produce induces a structure of JB$^*$-triple on $H$. 

Moreover, every JB$^*$-algebra $B$ (see, e.g. \cite{Wright1977} or \cite[Section 3.8]{hanche1984jordan}) is a JB$^*$-triple under the triple product defined by $\J xyz = (x\circ y^*) \circ z + (z\circ y^*)\circ x - (x\circ z)\circ y^*$ ($x,y,z\in B$) (see \cite[Lemma 3.1.6]{chubook} or \cite[Theorem 4.1.45]{Cabrera-Rodriguez-vol1}). We recall that a JB$^*$-algebra is a complex Jordan Banach algebra $A$ equipped
with an algebra involution $^*$ satisfying the following three conditions $$\| a\circ b\|\leq \|a\|\,\|b\|,\ \|a^*\| = \|a\|,\hbox{ and }\|\J a{a^*}a \|= \|a\|^3,$$ for all $a,b\in
A$, where $\J a{a^*}a  = 2 (a\circ a^*) \circ a - a^2 \circ a^*$.\smallskip

A formidable result due to Kaup asserts that a linear bijection between JB$^*$-triples is a triple isomorphism if and only if it is an isometry (cf. \cite[Proposition 5.5]{kaup1983riemann}).\smallskip

Given $a,b\in E$ the symbol $Q(a,b)$ will stand for the conjugate-linear operator given by $Q(a,b)(x)=\J axb$. We shall write $Q(a)$ for $Q(a,a)$.\smallskip

An element $e$ in a JB$^*$-triple $E$ is said to be a \emph{tripotent} if $e= \J eee$. Every projection in a C$^*$-algebra $A$ is a tripotent when the latter is regarded as a JB$^*$-triple. Actually, tripotents in $A$ are precisely partial isometries.\smallskip

For each tripotent $e\in E$, the eigenvalues of the mapping $L(e,e)$ are contained in the set $\{0,\frac12,1\}$. Given $i\in\{0,1,2\}$, the linear operator $P_{i} (e) : E\to E$ is defined by $$P_2 (e) = L(e,e)(2 L(e,e) -id_{E})=Q(e)^2,$$ $$ P_1 (e) = 4
L(e,e)(id_{E}-L(e,e)) =2\left(L(e,e)-Q(e)^2\right),$$ $$ \ \hbox{ and } P_0 (e) =
(id_{E}-L(e,e)) (id_{E}-2 L(e,e)).$$ It is known that $P_0(e),$ $P_1(e)$ and $P_2(e)$ are contractive linear projections (see \cite[Corollary 1.2]{Friedman-Russo}), which are called the \emph{Peirce} \emph{projections} associated with $e$. 
Furthermore, the range of $P_i(e)$ is the eigenspace, $E_i(e)$, of $L(e, e)$ corresponding to the eigenvalue $\frac{i}{2},$ and
$$E= E_{2} (e) \oplus E_{1} (e)\oplus E_0 (e)$$ is known as the \emph{Peirce decomposition} of $E$ relative to $e$ (see \cite{Friedman-Russo}, \cite[Definition 1.2.37]{chubook} or \cite[\S 4.2.2]{Cabrera-Rodriguez-vol1} and \cite[\S 5.7]{Cabrera-Rodriguez-vol2} for more details).
If $E$ is a unital C$^*$-algebra and $e\in E$  a tripotent, then $e$ is a partial isometry with initial projection $p_i$ and final projection $p_f$. The Peirce projections are given by the following identities
$$  P_2(e)(x)=p_fxp_i\,, \,  P_1(e)(x)=p_fx(1-p_i)+(1-p_f)xp_i\,,\,   P_0(e)(u) = (1-p_f)x(1-p_i) \,,     $$
where $x$ runs through $E$. 

A tripotent $e$ is called \emph{complete} if $E_0 (e) =\{0\}$.
If $E= E_2(e)$, or equivalently, if $\{e,e,x\}={x}$ for all $x\in E$, we say that $e$ is \emph{unitary}.\smallskip

For each tripotent $e$ in a JB$^*$-triple, $E$, the Peirce-2 subspace $E_2 (e)$ is a unital JB$^*$-algebra with unit $e$,
product $a\circ_{e} b := \{ a,e,b\}$ and involution $a^{*_e} := \J eae$ (cf. \cite[\S 1.2 and Remark 3.2.2]{chubook}).
As we noticed above, every JB$^*$-algebra is a JB$^*$-triple with respect to the product
$$\{a,b,c\}=(a \circ b^*) \circ c +(c \circ b^*) \circ a - (a \circ c) \circ b^*.$$ By Kaup's theorem (see \cite[Proposition 5.5]{kaup1983riemann}) the triple product on $E_2 (e)$ is uniquely determined by the expression
\begin{equation}\label{eq product Peirce2 as JB*-algebra} \{ a,b,c\} =(a \circ_{e} b^{*_e}) \circ_{e} c +(c \circ_{e} b^{*_e}) \circ_e a - (a \circ_e c) \circ b^{*_e},
\end{equation}
for every $a,b,c\in E_2 (e)$. Therefore, unital JB$^*$-algebras are in one-to-one correspondence with JB$^*$-triples
admitting a unitary element.\smallskip

We shall make use of the following natural partial order defined on the set of tripotents in a JB$^*$-triple $E$. Two tripotents $e,v$ in $E$ are called \emph{orthogonal} (denoted by $e\perp v$) if $\{e,e,v\}=0$ ($\Leftrightarrow$ $\{v,v,e\}=0$ $\Leftrightarrow$ $e\in E_0(v)$ $\Leftrightarrow$ $v\in E_0(e)$). Suppose $e,u$ are tripotents in $E$, we say that $e\leq u$ if $u-e$ is a tripotent which is orthogonal to $e$. By \cite[Corollary 1.7]{Friedman-Russo}  $e\leq u$ if and only if any of the equivalent conditions holds:
\begin{enumerate}[$(a)$]
\item $P_2(e)(u)= e$;
\item $\{e,u,e\}=e$;
\item $e$ is a projection {\rm(}i.e. a self-adjoint idempotent{\rm)} in the JB$^*$-algebra $E_2(u)$.
\end{enumerate}

A \emph{JBW$^*$-triple} is a JB$^*$-triple which is also a dual Banach space. In the triple setting, JBW$^*$-triples play the role of von Neumann algebras in the class of C$^*$-algebras. A fundamental result in the theory of JB$^*$-triples proves that every JBW$^*$-triple admits a unique (isometric) predual and its product is separately weak$^*$-to-weak$^*$ continuous (see \cite{BaTi}). 
JBW$^*$-algebras, von Neumann algebras, and complex Hilbert spaces are examples of JBW$^*$-triples for the triple products presented above (cf. \cite[Example 2.5.33 and Lemma 3.1.6]{chubook}). In particular, if $V$ is a von Neumann algebra and $p\in V$ is a projection, then $pV$ is a JBW$^*$-subtriple of $V$ while $pVp$ is a von Neumann subalgebra (with unit $p$). \smallskip

The bidual, $E^{**},$ of every JB$^*$-triple, $E,$ is a JBW$^*$-triple whose triple product extends that on $E$, and thus the inclusion of $E$ into $E^{**}$ is a triple homomorphism (cf. \cite[Corollary 11]{Dineen1986seconddual} or \cite[Theorem 1.4]{BaTi}).\smallskip 

The complete tripotents of a JB$^*$-triple $E$ are precisely the extreme points of its closed unit ball (cf. \cite[Lemma 4.1]{braun1978holomorphic} and \cite[Proposition 3.5]{kaup1977jordan} or
\cite[Theorem 3.2.3]{chubook}).
Therefore every JBW$^*$-triple contains a huge set of complete tripotents.\smallskip

\begin{rem} Let us point out that the notation for Peirce subspaces uses the name of the given triple.
So, if we have a JB$^*$-triple $E$, the Peirce subspaces are denoted by $E_j(e)$ as above; for a JB$^*$-triple $M$ the Peirce subspaces are denoted by $M_j(e)$ etc. 
\end{rem}

The theory of JBW$^*$-triples is deeply indebted with the study on the predual of JBW$^*$-triples developed by F. Friedman and B. Russo in \cite{Friedman-Russo}. Among the many influencing results established in this reference, it is shown that for each non-zero functional $\varphi$ in the predual, $M_*$, of a JBW$^*$-triple $M$, there is a unique tripotent $s(\varphi)\in M$, called the \emph{support tripotent} of $\varphi$, such that $\varphi=\varphi\circ P_2(s(\varphi))$, and $\varphi|_{M_2(s(\varphi))}$ is a faithful positive functional on the JBW$^*$-algebra $M_2(s(\varphi))$ (cf. \cite[Proposition 2]{Friedman-Russo}, or \cite[Proposition 5.10.57]{Cabrera-Rodriguez-vol2}).  We recall that a functional $\varphi$ in the dual space of a JB$^*$-algebra $B$ is called faithful if $\varphi (a) = 0$ for $a\geq 0$ implies $a=0$. We know from 
\cite[part $(b)$ in the proof of Proposition 2]{Friedman-Russo} that \begin{equation}\label{eq order in support tripotents}\hbox{if $u$ is a tripotent in $M$ with $\|\varphi\| = \varphi (u)$, then $u \geq s(\varphi)$.}
\end{equation}

Note that if $M$ is a von Neumann algebra, $\varphi\in M_*$ and $u=s(\varphi)$, then $u$ is a partial isometry and, moreover, the functional
$$u\varphi:x\mapsto \varphi(ux)$$ 
is positive and usually denoted by $\abs{\varphi}$. Then
$\varphi=u^*\abs{\varphi}$ is the polar decomposition of $\varphi$ (see \cite[Theorem III.4.2 and Definition III.4.3]{Tak}).

It is now time to recall the definition of the pre-Hilbert semi-norm appearing in Grothendieck's inequalities, which were introduced by J.T. Barton and Y. Friedman in \cite{barton1987grothendieck}. Suppose $\varphi$ is a functional in the predual of JBW$^*$-triple $M$. By \cite[Proposition 1.2]{barton1987grothendieck} the mapping $M\times M\to \mathbb{C}$, $(x,y)\mapsto \varphi\{x,y,s(\varphi)\}$ is a positive semi-definite sesquilinear form on $M$. In particular, the Cauchy-Schwarz inequality holds. The associated pre-Hilbert semi-norm is denoted by $\norm{x}_{\varphi}:=(\varphi\{x,x,s(\varphi)\})^{1/2}$ ($x\in M$). It is further known that $$\|x\|_
{\varphi}^2= \varphi\{x,x,s(\varphi)\}=\varphi\{x,x,z\},$$ whenever $z$ is an element in $M$ satisfying $\varphi(z)=\norm{\varphi}=\norm{z}=1$. In particular, $\|x\|_
{\varphi}^2= \varphi\{x,x,u\}$ for every tripotent $u\in M$ with $u\geq s(\varphi)$. Moreover, as a consequence of the fact that $\|\{x,y,z\}\|\le\|x\| \|y\| \|z\|$ for all $x,y,z$ in a JB$^*$-triple, we get \begin{equation}\label{eq semi-norm inequality} \|x\|_\varphi\le\sqrt{\|\varphi\|}\|x\|.\end{equation}\smallskip

\subsection{A representation of JBW$^*$-triples}\label{s:represent}

Key tools we use to prove our main results include  structure results of JBW$^*$-triples obtained by G. Horn and E. Neher in  \cite[(1.7)]{horn1987classification}, \cite[(1.20)]{horn1988classification}, and recently revisited in \cite{hamhalter2019mwnc} to decompose every JBW$^*$-triple $M$ in a suitable way. Before formulating the variant we need to give the following easy lemma on
decomposing special JBW$^*$-triples.

\begin{lemma}\label{L:pV decomposition}
Let $V$ be a von Neumann algebra, $p\in V$ a projection and $(z_j)_{j\in J}$ an orthogonal family of projections in the center of $pVp$ with sum equal to $p$.
Then
$$pV=\bigoplus_{j\in J}^{\ell_\infty} z_jV.$$
More precisely, the mapping
$$L:x\mapsto (z_jx)_{j\in J}$$
is an onto isometry witnessing the above equality.
\end{lemma}

\begin{proof}
The mapping $L$ is clearly a one-to-one linear mapping with $\norm{L}\le 1$.
Moreover, for any $a,b,c\in pV$ and $j\in J$ we have
$$\J {z_ja}{z_jb}{z_jc}=\frac12(z_jab^*z_jc+z_jcb^*z_ja)=\frac12(z_j a b^* c + z_j c b^* a) =z_j \J abc,$$
where in the second equality we used the fact that the elements $ab^*$ and $cb^*$ belong to $pVp$ and hence they commute with $z_j$.\smallskip

It follows that $L$ is a triple homomorphism. Since $L$ is injective, it is an isometry by \cite[Theorem 3.4.1]{chubook}.\smallskip

Finally, it is clear that the range contains all elements with only finitely many nonzero coordinates. Since $L$ is weak$^*$-to-weak$^*$ continuous, it follows that $L$ is onto.
\end{proof}

The promised representation result follows. For definitions and basic results on types of projections in von Neumann algebras we refer to \cite[Chapter V]{Tak}.

\begin{prop}\label{P:representation}
 Let $M$ be any JBW$^*$-triple. Then $M$ is (isometrically) JB$^*$triple isomorphic to a JBW$^*$-triple of the form
$$ \left( \bigoplus_{k\in\Lambda}^{\ell_{\infty}} L^\infty(\mu_k,C_k)\right)\oplus^{\ell_{\infty}} N\oplus^{\ell_{\infty}} p_1 V \oplus^{\ell_{\infty}} p_2V \oplus^{\ell_{\infty}}p_3V,$$
where
\begin{enumerate}[$\bullet$]
    \item $(\mu_k)_{k\in\Lambda}$ is a {\rm(}possibly empty{\rm)} family of probability measures;
    \item Each $C_k$ is a finite dimensional JB$^*$-triple {\rm(}actually a finite dimensional Cartan factor{\rm)} for any $k\in\Lambda$;
    \item $N$ is a JBW$^*$-algebra;
    \item $V$ is a von Neumann algebra, $p_1,p_2,p_3\in V$ are projections such that $p_1$ is  properly infinite, $p_2Vp_2$ is a von Neumann algebra of type $II_1$  and $p_3Vp_3$ is a finite von Neumann algebra of type $I$.
\end{enumerate} 
\end{prop}

\begin{proof}
By \cite[Proposition 9.2]{hamhalter2019mwnc}  $M$ is (isometrically) JB$^*$triple isomorphic to a JBW$^*$-triple of the form
\begin{equation*}
\left( \bigoplus_{k\in\Lambda}^{\ell_{\infty}} L^\infty(\mu_k,C_k)\right)\oplus^{\ell_{\infty}} N\oplus^{\ell_{\infty}} pV,
\end{equation*}
where $(\mu_k)_{k\in\Lambda}$, $(C_k)_{k\in\Lambda}$ and $N$ have the properties given in the statement and, moreover, $V$ is a von Neumann algebra and $p\in V$ is a projection.

It remains to refine this decomposition a bit. The summand $p V$ can be decomposed as a direct sum of two summands of the form $p_1 V$ and $p'_2 V$, where $p_1$ is a properly infinite projection and $p'_2$ is a finite projection (cf. \cite[Proposition 6.3.7]{KR2} or \cite[Theorem 10.1]{hamhalter2019mwnc}).

Further, by \cite[Theorem V.1.19]{Tak} there are
orthogononal  central projections $z_1,z_2$ in $p'_2Vp'_2$ with $z_1+z_2=p$ such that $z_1p'_2Vp'_2$ is  of type $I$  and $z_2p'_2Vp'_2$ of type $II_1$.
To complete the proof set $p_2=z_2p'_2$, $p_3=z_1p'_2$ and use Lemma~\ref{L:pV decomposition}.
\end{proof}

\section{Little Grothendieck inequality}\label{sec:little}

The difficulties around Barton-Friedman conjecture are essentially due to a gap in the proof of the so-called \emph{little Grothendieck inequality} stated in \cite[Theorem 1.3]{barton1987grothendieck}. As pointed out in \cite{peralta2001grothendieck} only the following statement was actually proved.

\begin{lemma}\label{lemma little w* norm attaining}{\rm(\cite[Lemma 3]{peralta2001grothendieck}, \cite[Theorem 1.3]{barton1987grothendieck})}
Let $M$ be a complex JBW$^*$-triple, $H$ a complex Hilbert space, and let $T : M\to H$ be a norm-attaining weak$^*$-to-weak continuous linear operator. Then there exists a norm-one normal functional $\varphi\in M_{*}$ satisfying $$\|T(x)\| \leq \sqrt{2}\  \|T\| \  \|x\|_{\varphi},$$ for all $x \in M$. $\hfill\Box$
\end{lemma}

In \cite{peralta2001grothendieck} it was observed that the assumption of norm-attaining, tacitly used in \cite{barton1987grothendieck}, need not to be satisfied. Via approximating operators by norm-attaining ones the following 
perturbed version of \cite[Theorem 1.3]{barton1987grothendieck}
was proved.

\begin{thm}\label{t little complex w*}\cite[Theorem 3]{peralta2001grothendieck}
Let $K>\sqrt{2}$ and $\varepsilon>0$. Then, for every JBW$^*$-triple $M$, every complex Hilbert space $H$, and
every weak$^*$-to-weak continuous linear operator $T:M \rightarrow
{H}$, there exist norm-one functionals $\varphi_{1},\varphi_{2}
\in M_{*}$ such that the inequality $$\|T(x)\| \leq K \, \|T\|
\, \sqrt{\|x\|_{\varphi_1}^2+\varepsilon\norm{x}_{\varphi_{2}}^2}$$ holds for all $x\in
M$. $\hfill\Box$
\end{thm}

This version is enough for many structure results on JBW$^*$-triples, but the question whether the perturbation is necessary, remained to be challenging.
We can get rid of the perturbation if we assume that the JBW$^*$-triple $M$ contains a unitary element, or equivalently, when $M$ is a (unital) JBW$^*$-algebra, as witnessed by the following theorem.

\begin{thm}\label{t LG JB*algebras}\cite[Theorem 4]{peralta2001grothendieck}
Let $K>2$ and let $M$ be a JBW$^*$-triple admitting a unitary element $u$.
Then for every complex Hilbert space and every weak$^*$-to-weak continuous
linear operator $T: M \to H$ there exists a norm-one functional
$\varphi\in M_{*}$ such that $s(\varphi)\leq u$ and $$\|T(x)\| \leq K \, \|T\| \,
\|x\|_{\varphi},$$ for all $x\in M$.  $\hfill\Box$
\end{thm}

We are going to extend this theorem to general JBW$^*$-triples by analyzing the behaviour of the seminorms
$$\norm{\cdot}_{\varphi_1,\varphi_2}:=\sqrt{\norm{\cdot}_{\varphi_1}^2+\norm{\cdot}_{\varphi_2}^2}$$
for a pair of normal functionals which do not necessarily have norm one. More specifically, we are going to prove the following theorem.

\begin{thm}\label{ttt solution to LG problem JBW*-triples} Let $M$ be a JBW$^*$-triple. Then given any two functionals $\varphi_1,\varphi_2$ in $M_*$, there exists a norm-one functional $\psi\in M_*$ such that
$$\norm{x}_{\varphi_1,\varphi_2}\le {\sqrt2\cdot\sqrt{\norm{\varphi_1}+\norm{\varphi_2}}}\cdot \norm{x}_\psi,$$ for all $x\in M.$ Furthermore, given $K> { 2}$, for every complex Hilbert space $H$, and every weak$^*$-to-weak continuous linear operator $T:M \rightarrow {H}$, there exists a norm-one functional $\psi\in M_{*}$ satisfying  $$\|T(x)\| \leq K \, \|T\| \, \|x\|_{\psi}$$  for all $x\in
M$.
\end{thm}

This theorem will be proved in Theorem~\ref{tt solution to LG problem JBW*-triples} below.

Observe that, once we establish the first estimate in this theorem, the second part follows easily from Theorem \ref{t little complex w*} (note that $\sqrt{\|x\|_{\varphi_1}^2+\varepsilon\norm{x}^2_{\varphi_{2}}}=\norm{x}_{\varphi_1,{\varepsilon}\varphi_2}$).

The first estimate will be proved using the representation from Proposition~\ref{P:representation}. We will prove it for individual summands and then we will glue the results together using the following proposition which is a finer version of \cite[Theorem 2.12]{peralta2005new}.

\begin{prop}\label{p stability of the LG under ellinfty sums precise}
Let $\{ M_{\alpha}\}_{\alpha\in \Lambda}$ be a family of JBW$^*$-triples for which there exists a positive constant $G$ satisfying that for every $\alpha\in \Lambda,$ and every couple of  normal functionals $\varphi_{1,\alpha},\varphi_{2,\alpha}\in (M_{\alpha})_{*}$ there exists a norm-one functional $\varphi_{\alpha} \in (M_{\alpha})_{*}$
satisfying $$ \| x \|_{\varphi_{1,\alpha},\varphi_{2,\alpha}} \leq G \sqrt{\norm{\varphi_{1,\alpha}}+\norm{\varphi_{2,\alpha}}}\, \|x\|_{\varphi_{\alpha}},$$ for all $x\in M_{\alpha}$. Let $M = \bigoplus^{\ell_{\infty}}_{\alpha\in
\Lambda} M_{\alpha}$.  Then for every couple of normal functionals $\varphi_{1},\varphi_{2}\in M_{*}$ there exists a norm-one functional $\varphi \in M_{*}$ satisfying $$\|x\|_{\varphi_1,\varphi_2} \leq  G \sqrt{\norm{\varphi_{1,\alpha}}+\norm{\varphi_{2,\alpha}}} \, \|x\|_{\varphi},$$ for all
$x\in M$.
\end{prop}

\begin{proof}
Let $\varphi_1,\varphi_2\in M_*$ be given. For $\alpha\in\Lambda$ and $j=1,2$ denote by $\varphi_{j,\alpha}$ the restriction of $\varphi_j$ to $M_\alpha$ (or, more precisely, the composition of $\varphi_j$ with the canonical embedding of $M_\alpha$ into $M$). By the assumption there is
a norm-one functional $\varphi_\alpha\in (M_\alpha)_*$
with
$$ \| x \|_{\varphi_{1,\alpha},\varphi_{2,\alpha}} \leq G \sqrt{\norm{\varphi_{1,\alpha}}+\norm{\varphi_{2,\alpha}}}\, \|x\|_{\varphi_{\alpha}},\mbox{ for }x\in M_\alpha.$$
Further, set
$$c_\alpha =\frac{\norm{\varphi_{1,\alpha}}+\norm{\varphi_{2,\alpha}}}{\norm{\varphi_1}+\norm{\varphi_2}}, \quad \alpha\in\Lambda,$$
and observe that $\sum_{\alpha\in \Lambda}c_\alpha=1$. Thus the functional $\varphi\in M_*$ defined by
$$\varphi((x_\alpha)_{\alpha\in\Lambda})=\sum_{\alpha\in\Lambda}c_\alpha\varphi_\alpha(x_\alpha)
\mbox{ for } x=(x_\alpha)_{\alpha\in\Lambda}\in M,$$
has norm one. Moreover, for each $x\in M$ we have
$$\begin{aligned}
\norm{x}_{\varphi_1,\varphi_2}^2&=
\sum_{\alpha\in\Lambda}\norm{x_\alpha}_{\varphi_{1,\alpha},\varphi_{2,\alpha}}^2 \le \sum_{\alpha\in\Lambda} G^2 (\norm{\varphi_{1,\alpha}}+\norm{\varphi_{2,\alpha}})\norm{x_\alpha}_{\varphi_\alpha}^2 \\&= G^2 (\norm{\varphi_1}+\norm{\varphi_2}) \sum_{\alpha\in\Lambda} c_\alpha \norm{x_\alpha}_{\varphi_\alpha}^2=G^2 (\norm{\varphi_1}+\norm{\varphi_2})\norm{x}_\varphi^2.\end{aligned}$$
\end{proof}

The individual summands will be addressed in the three following sections, in the last section we glue the results together and show that a solution to the Barton--Friedman conjecture follows.

The proof for the summands $N$ and $p_1V$ is given in Corollary~\ref{c LG when Peirce-2 are upward directed} and it is done by a refinement of the proof of Theorem~\ref{t LG JB*algebras} using some ideas from \cite{hamhalter2019mwnc}. The proof for the remaining cases is done by showing that in these cases any seminorm of the form $\norm{\cdot}_{\varphi_1,\varphi_2}$ attains its maximum on $B_M$ and then applying Lemma~\ref{lemma little w* norm attaining}. The last step of this approach is explained in the following lemma.

\begin{lemma}\label{l:double NA -> estimate}
Let $\varphi_1,\varphi_2\in M_*$ be two normal functionals such that the seminorm  $\norm{\cdot}_{\varphi_1,\varphi_2}$ attains its maximum on $B_M$. Then there is a norm-one functional $\psi\in M_*$ such that
$$\|x\|_{\varphi_1,\varphi_2}\le\sqrt{2}\sqrt{\|\varphi_1\|+\|\varphi_2\|}\cdot\|x\|_\psi$$ for all $x\in M$.
\end{lemma}

\begin{proof}
 Set $$N_{\varphi_1,\varphi_2}=\{x\in M : \|x\|_{\varphi_1,\varphi_2} =0\}.$$ 
 On the quotient space $M/N_{\varphi_1,\varphi_2}$,
the semi-norm $\|\cdot\|_{\varphi_1,\varphi_2}$ becomes a pre-Hilbert norm.
Let $H_{\varphi_1,\varphi_2}$ be the completion of the so-defined pre-Hilbert space and let $\pi_{_{\varphi_1,\varphi_2}}$ be the natural quotient map viewed as a map from $M$ into $H_{\varphi_1,\varphi_2}$.
The separate weak$^*$-to-weak$^*$ continuity of the triple product and \eqref{eq semi-norm inequality} ensure that $\pi_{_{\varphi_1,\varphi_2}}$ is a weak$^*$-to-weak continuous linear operator with norm at most $\sqrt{\|\varphi_1\|+\|\varphi_2\|}$.
Finally, we may apply  Lemma \ref{lemma little w* norm attaining}  to the operator $T=\pi_{_{\varphi_1,\varphi_2}}$. 
\end{proof}

\section{JBW$^*$-triples in which Peirce-2 subspaces of tripotents are upward directed}

In this section we particularize our study to JBW$^*$-triples satisfying that Peirce-2 subspaces of tripotents are upward directed by inclusion. The idea stems from \cite{hamhalter2019mwnc} where such JBW$^*$-triples were considered in order to have a mild substitute for the lack of an order, see e.g. \cite[Proposition 6.5]{hamhalter2019mwnc}. Let us begin with a series of technical lemmata.

\begin{lemma}\label{l 1.5} Let $\varphi_1,\varphi_2$ be two  functionals in the predual of a JBW$^*$-triple $M$. Suppose there exists a tripotent $p$ in $M$ such that $s(\varphi_1)\le p$ and $s(\varphi_2)\le p$. Then the functional 
 $\psi=\frac{\varphi_1+\varphi_2}{\norm{\varphi_1}+\norm{\varphi_2}}$
satisfies  $\norm{\psi}=1$, $s(\psi)\le p,$ and
$$\norm{x}_{\varphi_1,\varphi_2}=
\sqrt{\norm{\varphi_1}+\norm{\varphi_2}}\cdot\norm{x}_\psi,\quad x\in M.$$
\end{lemma}

\begin{proof} Set $e=s(\varphi_2)$ and $u=s(\varphi_1)$. By the assumption we have $u\le p$ and $e\le p$. Further, $\varphi_2(p)=\varphi_2(e)={\norm{\varphi_2}}$ and $\varphi_1(p)=\varphi_1(u)={\norm{\varphi_1}}$, so $\psi(p)=1$. Since clearly $\norm{\psi}\le 1$, we deduce that $\norm{\psi}=\psi(p)=1$ and hence
$s(\psi)\le p$ (cf. \eqref{eq order in support tripotents}).\smallskip

Finally, for $x\in M$ we have
 $$\begin{aligned}
\norm{x}_{\psi}^2&=\psi(\J xx{s(\psi)}) =\psi(\J xxp)
=\frac{\varphi_1(\J xxp)+\varphi_2(\J xxp)}{\norm{\varphi_1}+\norm{\varphi_2}}\\&= \frac{\varphi_1(\J xxu)+\varphi_2(\J xxe)}{\norm{\varphi_1}+\norm{\varphi_2}}=\frac{\norm{x}_{\varphi_1,\varphi_2}^2}{\norm{\varphi_1}+\norm{\varphi_2}}.
\end{aligned}
$$
\end{proof}

In our next proposition we show that the semi-norm given by a normal functional whose support tripotent is contained in the Peirce-2 subspace of another tripotent $p$ in a JBW$^*$-triple $M$ can be bounded by the semi-norm given by a positive functional in the predual of the JBW$^*$-algebra $M_2(p)$.

\begin{prop}\label{p 3} Let $M$ be a JBW$^*$-triple and let $\varphi\in M_*$.
Assume that $p\in M$ is a tripotent such that $s(\varphi)\in M_2(p)$.
Then there exists a  functional $\tilde{\varphi}\in M_*$ such that {$\norm{\tilde{\varphi}}=\norm{\varphi}$}, 
$s(\tilde{\varphi})\le p$ and $\norm{x}_\varphi\le\sqrt{2}\norm{x}_{\tilde{\varphi}}$ for all $x\in M$.
\end{prop}

\begin{proof}
We mimic the approach in the proof of \cite[Lemma 7.7]{hamhalter2019mwnc}. By the arguments in the first paragraph in the proof of \cite[Proposition 2.4]{BuFPMaMoPe} (see also \cite[Lemma 3.9]{EdMaFPHoPe}) we can find a unital JB$^*$-algebra $B$ and an isometric triple embedding $\pi$ of $M$ into $B$ such that $\pi(p)$ is a projection in $B$. We can therefore assume that $M$ is a JB$^*$-subtriple of $B$ and $p$ is a projection in $B$. The triple product in $B$ (and in $M$) is uniquely determined by the expression $\{a,b,c\}=(a \circ b^*)\circ c + (c\circ b^*)\circ a - (a\circ c)\circ b^*$ ($a,b,c\in B$).\smallskip

Set $u=s(\varphi)$. Define $G:B\to B$ by $G(x)=P_2(u)(x\circ u)$. Having in mind that $1-p\perp u$ (because $1-p\in M_0(p)$ and $u\in M_2(p)$), and hence for each $x\in M$, we have $P_2(u) (x\circ u) = P_2(u) \{x,1,u\} =P_2(u) \{x,p,u\} \in M$, we deduce that $G$ maps $M$ into $M_2(u)$ and its restriction to $M$ is weak$^*$-to-weak$^*$ continuous.\smallskip

Set $\tilde{\varphi}=\varphi\circ G|_M$. Then $\tilde{\varphi}\in M_*$ and $\norm{\tilde{\varphi}}\le\norm{\varphi}$ (as Peirce projections are contractive and hence clearly $\norm{G}\le 1$).\smallskip

Moreover,
$$\tilde{\varphi}(p)=\varphi(P_2(u)(p\circ u))=\varphi P_2(u) \{p,p,u\}=\varphi(u)=\norm{\varphi},$$
hence $\norm{\tilde{\varphi}}=\norm{\varphi}$ and $s(\tilde{\varphi})\le p$ (see \eqref{eq order in support tripotents}).\smallskip

Finally, it is explicitly shown in the proof of \cite[Lemma 7.7]{hamhalter2019mwnc} that for each $x\in B$ we have $P_2(u)(\J xxu + \J {x^*}{x^*}u)=2 G(P_2(p)\J xxp),$ and hence
$$\begin{aligned}\norm{x}_{\tilde{\varphi}}^2&=\tilde{\varphi}(\J xx{s(\tilde{\varphi})})
=\tilde{\varphi}(\J xxp)=\tilde{\varphi}(P_2(p)\J xxp)=
\varphi(G(P_2(p)\J xxp))
\\&=\frac12\varphi(P_2(u) (\J xxu+\J{x^*}{x^*}u))\ge \frac12\norm{x}_\varphi^2.\end{aligned}$$
Note that in the third equality we used that $\tilde{\varphi}=\tilde{\varphi}\circ P_2(p)$. This follows from the fact that  $s(\tilde{\varphi})\le p$ which implies  $P_2(s(\tilde{\varphi}))=P_2(s(\tilde{\varphi}))P_2(p)$ (for example using \cite[Proposition 6.5]{hamhalter2019mwnc}).

This completes the argument.
\end{proof}

We can next combine Lemma \ref{l 1.5} and Proposition \ref{p 3} to obtain a strengthened conclusion.

\begin{prop}\label{p 4} Let $\varphi_1,\varphi_2$ be two  functionals in the predual of a JBW$^*$-triple $M$.
Assume there exists a tripotent $p\in M$ such that $\{s(\varphi_1),s(\varphi_2)\}\subseteq M_2(p)$.
Then there is a norm-one functional $\psi\in M_*$ such that $s(\psi)\le p$ and
$$\norm{x}_{\varphi_1,\varphi_2}\le {\sqrt2\cdot\sqrt{\norm{\varphi_1}+\norm{\varphi_2}}}\cdot \norm{x}_\psi,$$ for all $x\in M.$ 
\end{prop}

\begin{proof} Find, via Proposition \ref{p 3}, two  functionals $\tilde{\varphi}_1$ and $\tilde{\varphi_2}$ in $M_*$ such that $\norm{\tilde{\varphi_j}}=\norm{\varphi_j}$, $s(\tilde{\varphi_j})\le p$ and $\norm{x}_{\varphi_j}\le\sqrt{2}\norm{x}_{\tilde{\varphi}_j}$ for all $x\in M$ and $j\in \{1,2\}$. Take $\psi=\frac{\tilde{\varphi}_1+\tilde{\varphi_2}}{\norm{\varphi_1}+\norm{\varphi_2}}\in M_*$. Lemma \ref{l 1.5} implies that $s(\psi)\le p$ and
$$\norm{x}_{\varphi_1,\varphi_2}^2 = \norm{x}_{\varphi_1}^2+ \norm{x}_{\varphi_2}^2 \le 2 \left(\norm{x}_{\tilde{\varphi}_1}^2+ \norm{x}_{\tilde{\varphi}_2}^2 \right) = {2(\norm{\varphi_1}+\norm{\varphi_2})} \norm{x}_\psi^2,$$ for all $x\in M.$\smallskip
\end{proof}

\begin{cor}\label{c LG when Peirce-2 are upward directed} Let $M$ be a JBW$^*$-triple in which Peirce-2 subspaces of tripotents are upward directed by inclusion. Then given any $\varphi_1,\varphi_2$ in $M_*$, there exists a norm-one functional $\psi\in M_*$ such that
$$\norm{x}_{\varphi_1,\varphi_2}\le {\sqrt2 \sqrt{\norm{\varphi_1}+\norm{\varphi_2}}} \norm{x}_\psi,$$ for all $x\in M.$ This holds, in particular, when $M$ is either a JBW$^*$-algebra or a JBW$^*$-triple of the form $pV$, where $V$ is a von Neumann algebra and $p\in V$ is a properly infinite projection.
\end{cor}

\begin{proof} The first statement in a straight consequence of the previous Proposition \ref{p 4}. 
The second statement follows from \cite[Remark 9.13]{hamhalter2019mwnc}.
\end{proof}

\section{Finite dimensional Cartan factors} 

In this section we shall deal with JBW$^*$-triples of the form $L^{\infty}(\mu, C)$, where $\mu$ is a probability measure and $C$ is a finite dimensional Cartan factor. In fact, the results work in a slightly more general setting -- if $C$ is a finite-dimensional JB$^*$-triple.  Henceforth, let $C$ be such a JB$^*$-triple.  Since $C$ is finite dimensional, every bounded linear operator from $C$ into a Hilbert space attains its norm.
In particular, any seminorm $\norm{\cdot}_{\varphi_1,\varphi_2}$ attains its maximum on the unit ball $B_C$. We will show that this property can be carried over to the space $L^\infty(\mu,C)$.
This  goal  will be obtained after a series of lemmata.

\begin{lemma}\label{l (B1)} The mapping $C\times C^*\to [0,\infty)$, $(x,\varphi)\mapsto\norm{x}_\varphi$ is continuous.
\end{lemma}

\begin{proof} The set $$A:=\{(\varphi,e)\in C^*\times S_C: \varphi(e)=\norm{\varphi}, \J eee=e\}$$ is clearly closed. Moreover, the mapping $\Phi:C\times A\to[0,\infty)$ given
by
$$\Phi(x,\varphi,e)=\varphi\J xxe$$
is continuous and $\Phi(x,\varphi,e)=\norm{x}_\varphi^2$ for $x\in C$ and $(\varphi,e)\in A$.\smallskip

Assume now that $(x_n,\varphi_n)$ is a sequence in $C\times C^*$ converging to an element $(x,\varphi)$. We will show that $\norm{x_n}_{\varphi_n}\to\norm{x}_\varphi$. Otherwise, up to passing to a subsequence, we may assume that $\norm{x_n}_{\varphi_n}\to c\ne\norm{x}_\varphi$
(note that the sequence $(\norm{x_n}_{\varphi_n})_n$ is bounded). Let $e_n=s(\varphi_n)$ for $n\in\en$. We may assume, without loss of generality, that the sequence $(e_n)$ converges to some $e\in C$. Since $(\varphi_n,e_n)\in A$ for each $n$, necessarily $(\varphi,e)\in A$ as well. Thus
$$\norm{x_n}_{\varphi_n}^2=\Phi(x_n,\varphi_n,e_n)\to\Phi(x,\varphi,e)=\norm{x}_\varphi^2,$$
a contradiction which completes the proof.
\end{proof}

\begin{lemma}\label{l (C1)} The set valued mapping $\Psi:(C^*)^2\to 2^{B_C}$ defined by
$$\Psi(\varphi_1,\varphi_2)=\{x\in B_C\setsep \norm{x}_{\varphi_1,\varphi_2}=\max_{y\in B_C} \norm{y}_{\varphi_1,\varphi_2}\}$$
is upper semi-continuous and compact-valued. Consequently, there is a Borel-measurable selection $H$ from $\Psi$.
\end{lemma} 

\begin{proof} Taking into account that $S_{C^*}$ is compact, by \cite[Lemma 3.1.1]{fabian-book} it is enough to show that the set
$$\{(\varphi_1,\varphi_2,x)\in (C^*)^2\times B_C\setsep
\norm{x}_{\varphi_1,\varphi_2}=\max_{y\in B_C} \norm{y}_{\varphi_1,\varphi_2}\}$$
is closed. But this easily follows from Lemma~\ref{l (B1)} as this set equals
$$\bigcap_{y\in B_C}\{(\varphi_1,\varphi_2,x)\in (C^*)^2\times B_C\setsep
\norm{x}_{\varphi_1,\varphi_2}\ge \norm{y}_{\varphi_1,\varphi_2}\}.$$

Since $\Psi$ has clearly nonempty values, the final statement follows, for example, from the Kuratowski-Ryll-Nardzewski theorem (see \cite[Theorem 18.13]{aliprantisborder}).
\end{proof}

Let $(\Omega,\Sigma,\mu)$ be a probability space, and let $M=L^\infty(\mu,C)$. Then $M$ is a JBW$^*$-triple (with the triple product defined pointwise) and $M_*=L^1(\mu,C^*)$.\smallskip

We need a more concrete description of the elements in $M_*$. Assume $g\in M_*=L^1(\mu,C^*)$. Let $u=s(g)$. Then $u$ is a tripotent in $M$, hence $u(\omega)$ is a tripotent in $C$ for almost all $\omega\in\Omega$. Under these circumstances we have

$$\begin{aligned}
\norm{g}&=g(u)=\Re g(u)=\Re \int \ip{g(\omega)}{u(\omega)}\di\mu(\omega)
= \int \Re \ip{g(\omega)}{u(\omega)}\di\mu(\omega)
\\&\le  \int \abs{\ip{g(\omega)}{u(\omega)}}\di\mu(\omega)
\le \int \norm{g(\omega)}\cdot\norm{u(\omega)}\di\mu(\omega)\le
 \int \norm{g(\omega)}\di\mu(\omega)=\norm{g}\end{aligned}$$
So, we have everywhere equalities, hence $\ip{g(\omega)}{u(\omega)}=\norm{g(\omega)}$ almost everywhere, and thus
$u(\omega)\ge s(g(\omega))$ almost everywhere (cf. \eqref{eq order in support tripotents}).\smallskip

It follows that for almost all $\omega$ we have
$$\norm{x}_{g(\omega)}^2=\ip{g(\omega)}{\J xx{u(\omega)}},\hbox{ for all } x\in C.$$
Therefore, given $f\in M$ we have
$$\norm{f}_g^2= \ip{g}{\J ffu}=\int \ip{g(\omega)}{\J{f(\omega)}{f(\omega)}{u(\omega)}}\di\mu(\omega)=
\int \norm{f(\omega)}_{g(\omega)}^2\di\mu(\omega).$$

Let $g_1,g_2\in L^1(\mu,C^*)$. Let $H$ be the Borel-measurable selection from $\Psi$ given by \ref{l (C1)}. We set $f(\omega)=H(g_1(\omega),g_2(\omega))$. Then $\norm{f}_\infty\le 1$. Let $h\in L^\infty(\mu,C)$ be any element of the unit ball. Then
$$\norm{h}_{g_1,g_2}^2=\int \norm{h(\omega)}_{g_1(\omega),g_2(\omega)}^2\di\mu(\omega)
\le \int \norm{f(\omega)}_{g_1(\omega),g_2(\omega)}^2\di\mu(\omega)
=\norm{f}_{g_1,g_2}^2.$$
Therefore the pre-Hilbert semi-norm $\norm{\cdot}_{g_1,g_2}$ attains its maximum on the closed unit ball of $L^\infty(\mu,C)$ (at $f$).\smallskip

The previous arguments combined with Lemma \ref{l:double NA -> estimate} provide the following solution to the little Grothendieck problem for JBW$^*$-triples of the form $L^\infty(\mu,C)$.

\begin{prop}\label{p LG Linfty mu C} Let $(\Omega,\Sigma,\mu)$ be a probability space, and let $M=L^\infty(\mu,C)$, where $C$ is a finite dimensional JB$^*$-triple.
Then for every couple of normal functionals $g_1,g_2\in M_*$ the pre-Hilbert semi-norm $\norm{\cdot}_{g_1,g_2}$ attains its maximum on the closed unit ball of $L^\infty(\mu,C)$, and thus there exists a norm-one functional $h\in M_*$ satisfying $$\|f \|_{g_1,g_2} \leq  \sqrt2\,\sqrt{\norm{g_1}+\norm{g_2}}\, \|f\|_{h},$$ for all $f\in M$.
\end{prop}

\section{Right ideals associated with finite projections in a von Neumann algebra} The aim of this section is to solve the little Grothendieck problem for  the summands $p_2V$ and $p_3V$ from Proposition~\ref{P:representation}. They require different methods, but some tools are common for both cases.
The first lemma shows how to express the hilbertian semi-norms using polar decomposition of the functional.

\begin{lemma}\label{L:shift phi to state}
Let $V$ be a von Neumann algebra, $p\in V$ a finite projection and $\varphi\in (pV)_*$. Then there is a positive functional $\psi$ on $pVp$ and a unitary element $u\in V$ such that $\norm{\psi}=\norm{\varphi}$, $\varphi(x)=\psi(xup)$ for $x\in pV$, $s(\psi) u^* = s(\varphi),$  and
$$ \norm{x}_\varphi^2=\frac12\left(\psi(xx^*)+\psi(pu^*x^*xup) \right) \mbox{ for  all }
 x\in pV.$$
\end{lemma}

\begin{proof}
Let $v=s(\varphi)$. Then $v$, being a tripotent in $pV$, is a partial isometry in $V$ with final projection $q\le p$. Denote by $r$ the initial projection. Further, since $p$ is finite, $q$ is finite as well, hence $v$ can be extended to a unitary operator $\tilde v\in V$ (cf. \cite[Proposition V.1.38]{Tak}).\smallskip

Set $\psi(x):=\varphi(x\tilde{v})$ for $x\in pV$. Since $x\mapsto x\tilde{v}$ is an isometry of $pV$ onto $pV$, we deduce that $\norm{\psi}=\norm{\varphi}$. Further, since
$$\psi(q)=\varphi(q\tilde{v})=\varphi(v)=\norm{\varphi}=\norm{\psi},$$
we deduce that $s(\psi)\le q$  (cf. \eqref{eq order in support tripotents}), hence $\psi|_{pVp}$ is a positive functional on $pVp$. It remains to observe that one can take $u=\tilde{v}^*$. Indeed, for any $x\in pV$ we have
$$\psi(x u p ) = \psi(x\tilde{v}^*p)=\psi(qx\tilde{v}^*pq)=\psi(qx\tilde{v}^*q)=\psi(qxr\tilde{v}^*)=\varphi(qxr)=\varphi(x).$$

In particular, $$\norm{\varphi}=\norm{\psi} = \psi (s(\psi))= \psi (s(\psi) p ) = \psi (s(\psi) u^* u p ) = \varphi (s(\psi) u^*),$$ which shows that $s(\psi) u^* \geq s(\varphi)$  (cf. \eqref{eq order in support tripotents}). But $s(\psi)\leq q$ implies that $s(\psi) u^* = s(\varphi)$.\smallskip

Finally, for any $x\in pV$ we have
$$\begin{aligned}\norm{x}_\varphi^2&=\varphi(\J xxv)=\frac12\varphi(xx^*v+vx^*x)=\frac12\psi(xx^*vup+vx^*xup)\\&=\frac12(\psi(xx^*)+\psi(pu^*x^*xup)),\end{aligned}$$
where in the last equality we used that $vup=v\tilde{v}^*p=qp=q$ and $s(\psi)\le q$ to obtain the first term and
$$\psi(vx^*xup)=\psi(qu^*x^*xup)=\psi(qpu^*x^*xup)=\psi(pu^*x^*xup),$$
to obtain the second term.
\end{proof}

The key result for algebras of type $II_1$ is established in the next lemma.

\begin{lemma}\label{L:max of double norm}
Let $V$ be a von Neumann algebra of type $II_1$ and let $p\in V$ be a projection. Then for each couple of  functionals $\varphi_1,\varphi_2\in (pV)_*$ the pre-Hilbert semi-norm $\norm{\cdot}_{\varphi_1,\varphi_2}$ attains its maximum on the closed unit ball of $pV$.
\end{lemma}

\begin{proof}
For $j=1,2$ let $\psi_j$ be a positive functional in $(pVp)_*$ and $u_j\in V$ a unitary element provided by Lemma~\ref{L:shift phi to state} for $\varphi_j$.
Then $$\norm{x}_{\varphi_1,\varphi_2}^2=\frac12(\psi_1(xx^*)+\psi_1(pu_1^*x^*xu_1p)+\psi_2(xx^*)+\psi_2(pu_2^*x^*xu_2p))$$
for any $x\in pV$.\smallskip

By the Krein-Milman theorem and the weak$^*$-compactness of the closed unit ball of $pV$, the supremum of this semi-norm on the closed unit ball of $pV$ is attained if and only if it is attained at an extreme point of this closed unit ball. Note that a tripotent in $pV$ is a partial isometry in $V$ with final projection below $p$, the tripotent is complete (i.e. it is an extreme point of the closed unit ball) if and only if its final projection equals $p$. Therefore the supremum of the semi-norm over the unit ball equals $\sqrt C$, where
$$\begin{aligned}C&=\sup\left\{\frac12(\psi_1(xx^*)+\psi_1(pu_1^*x^*xu_1p)+\psi_2(xx^*)+\psi_2(pu_2^*x^*xu_2p))\setsep xx^*=p\right\}\\
&=\sup\left\{\frac12(\psi_1(p)+\psi_2(p))+\frac12(\psi_1(pu_1^*x^*xu_1p)+\psi_2(pu_2^*x^*xu_2p))\setsep xx^*=p\right\}.\end{aligned}$$
Let $T$ be the center-valued trace on $V$ (cf. \cite[Theorem V.2.6]{Tak}). If $x\in V$ is such that $xx^*=p$, then $0\le x^*x\le 1$ and $T(x^*x)=T(p)$.
Hence
$$C\le \sup\left\{\frac12(\psi_1(p)+\psi_2(p))+\frac12(\psi_1(pu_1^*yu_1p)+\psi_2(pu_2^*yu_2p))\setsep \begin{array}{c}
    0\le y\le 1, \\  T(y)=T(p)  \end{array}\right\}.$$
The supremum on the right-hand side is attained, as it is a supremum of an affine weak$^*$-continuous functional over the convex weak$^*$-compact set
$$K=\left\{y\in V\setsep 0\le y\le 1, T(y)=T(p)\right\}.$$
So, the supremum is attained at an extreme point of $K$.
Now, we claim that every extreme point of $K$ is a projection.
Indeed, assume that, say, $y\in K$ is not a projection. Since $0\le y\le 1$, we may consider the spectral measure $E$ of $y$. Since $y$ is not a projection, there is some $\delta\in(0,\frac12)$ such that $q=E([\delta,1-\delta])\ne0$. Since $V$ is of type $II_1$, there is a projection $r\le q$ with $r\sim q-r$. Set
$$v=y+\delta(2r-q),\quad w=y-\delta(2r-q).$$
Then $y=\frac12(v+w)$, $T(v)=T(w)=T(y)=T(p)$ (as $T(r)=\frac12T(q)$ by \cite[Corollary V.2.8]{Tak}). Moreover
$$v\ge y-\delta q\ge 0,\mbox{ and }v\le y+\delta q\le 1,$$
and similarly for $w$. It follows that $v,w\in K$, so $y$ is not an extreme point of $K$. This finishes the proof of the claim.\smallskip

Fix $y\in\ext K$ where the supremum is attained. Then $y$ is a projection satisfying $T(y)=T(p)$, so $y\sim p$  by \cite[Corollary V.2.8]{Tak}. Therefore there is $x\in V$ with $xx^*=p$ and $x^*x=y$. We finally observe that the supremum $C$ is attained at this element $x$.
\end{proof}

The following technical lemma enables us, roughly speaking, to reduce the case $pV$ for a finite projection $p$ to the case $pV$ where the whole $V$ is finite.

{
\begin{lemma}\label{l 5.5 new}
Let $V$ be a von Neumann algebra and $p\le t$ two projections in $V$ such that $p$ is finite. Consider the JBW$^*$-triple $M=pV$ and its subtriple $N=pVt$. Let $\varphi_1,\varphi_2\in M_*$ be two  functionals such that $s(\varphi_j)\in N$ for $j=1,2$. Then
$$\sup\{\norm{x}_{\varphi_1,\varphi_2}\setsep x\in B_M\}=\sup\{\norm{x}_{\varphi_1,\varphi_2}\setsep x\in B_N\}.$$
\end{lemma}

\begin{proof}  We use some ideas from the proof of \cite[Proposition 2.8]{peralta2005new}.
 Let $W=tVt$. Then $W$ is a von Neumann algebra, a C$^*$-subalgebra of $V$ and $t$ is its unit. Set
$$u'_j=s(\varphi_j)\mbox{ for }j=1,2.$$
Both these tripotents are partial isometries in $W$ with final projection below $p$. Since $p$ is finite, by \cite[Proposition V.1.38]{Tak} these partial isometries can be extended to unitary elements $u''_1,u''_2\in W$.
Set
$$u_j=pu''_j\mbox{ for }j=1,2.$$
Then $u_1,u_2$ are partial isometries in $W$ with final projection equal to $p$. In particular, they are complete tripotents in $N$ and also in $M$.

Moreover,
$$u'_j\le u_j\mbox{ for }j=1,2,$$
where we use the standard order on tripotents. Indeed, it is enough to observe that
$$\J {u_j'}{u_j}{u_j'}= u_j'u_j^*u_j'=u_j' (u_j'')^*pu_j'=
u_j'(u_j'')^* p_f(u_j')u_j'=u_j'(u_j')^*u_j'=u_j'.$$
Further, define functionals
$\zeta_j\in W_*$ by
$\zeta_j(x)=\varphi_j(u_jx)$ for $x\in W$. Clearly $\norm{\zeta_j}\le \norm{\varphi_j}$ and, moreover,
$$ \zeta_j(t)=\varphi_j(u_j t)=\varphi_j(u_j)=\norm{\varphi_j},$$
hence $\zeta_j$ is positive (and $s(\zeta_j)\le t$).\smallskip

Given $x\in M$, set $x_1=xt$ and $x_2=x(1-t)$. Note that
$$\J xx{u_j}=\frac12(xx^*u_j+u_jx^*x)=\frac12(x_1x_1^*u_j+x_2x_2^*u_j+u_jx_1^*x_1+u_jx_1^*x_2)$$
 where we used that $x_1x_2^*=x_2x_1^*=0$ and $u_jx_2^*=0$ (the initial and the final projections of $u_j$ both are below $t$). Since $\frac12u_jx_1^*x_2\in pV(1-t)\subset M_1(u_j)$, we deduce
 $$P_2(u_j) \J xx{u_j} =\frac12P_2(u_j)(x_1x_1^*u_j+x_2x_2^*u_j+u_jx_1^*x_1)$$
 Using the fact that $s(\varphi_j)=u_j'\le u_j$ we infer that
 $$\begin{aligned}\norm{x}&_{\varphi_1,\varphi_2}^2=
 \frac12\varphi_1(x_1x_1^*u_1+x_2x_2^*u_1+u_1x_1^*x_1)+
 \frac12\varphi_2(x_1x_1^*u_2+x_2x_2^*u_2+u_2x_1^*x_1)
 \\& =\frac12(\varphi_1(x_1x_1^*u_1+x_2x_2^*u_1)+\zeta_1(x_1^*x_1)+
 \varphi_2(x_1x_1^*u_2+x_2x_2^*u_2)+\zeta_2(x_1^*x_1)).
 \end{aligned}$$

By the Krein-Milman theorem and the weak$^*$-compactness of $B_M$ (and $B_N$), the supremum of this semi-norm over any of these balls equals the supremum over its extreme points, i.e., over completes tripotents. Further note that a complete tripotent in $M$ (in $N$) is a partial isometry in $V$ (in $W$) with final projection equal to $p$, i.e, an element $x\in M$ ($x\in N$) satisfying $xx^*=p$. Since for $x\in M$ we have $xx^*=x_1x_1^*+x_2x_2^*$, we have
$$\begin{aligned}
\sup &\{\norm{x}^2_{\varphi_1,\varphi_2} \setsep x\in B_N\}\le \sup\{\norm{x}^2_{\varphi_1,\varphi_2}\setsep x\in B_M\}\\&
=\sup\{\norm{x}^2_{\varphi_1,\varphi_2}\setsep x\in M, xx^*=p\}
\\&=\frac12\sup\left\{\varphi_1(pu_1)+\varphi_2(pu_2)+\zeta_1(x_1^*x_1)+\zeta_2(x_1^*x_1)\setsep  x\in M, x_1x_1^*+x_2x_2^*= p\right\}
\\&\le\frac12\sup\left\{\varphi_1(pu_1)+\varphi_2(pu_2)+\zeta_1(y^*y)+\zeta_2(y^*y)\setsep  y\in N, yy^*\le p\right\}
\\&\le\frac12\sup\left\{\varphi_1(pu_1)+\varphi_2(pu_2)+\zeta_1(y^*y)+\zeta_2(y^*y)\setsep y\in B_N\right\}
\\&=\frac12\sup\left\{\varphi_1(pu_1)+\varphi_2(pu_2)+\zeta_1(y^*y)+\zeta_2(y^*y)\setsep y\in N, yy^*= p\right\}
\\&\le \sup\{\norm{x}^2_{\varphi_1,\varphi_2}\setsep x\in B_N\},
\end{aligned}$$
where we used that $y\mapsto(\zeta_1(y^*y)+\zeta_2(y^*y))^{1/2}$ is a weak$^*$-continuous pre-hilbertian semi-norm, hence the supremum can be computed over extreme points.
\end{proof}

}

We are now in a position to present a  solution to the little Grothendieck problem for
the summand $p_2V$ from Proposition~\ref{P:representation}.

\begin{prop}\label{p pV p finite type II1} Let $V$ be a von Neumann algebra and $p\in V$ a projection such that $pVp$ is of type $II_1$.
Then for any  $\varphi_1,\varphi_2\in (pV)_*$ { the semi-norm $\norm{\cdot}_{\varphi_1,\varphi_2}$ attains its maximum on the unit ball of $pV$ and therefore} there exists a norm-one functional $\psi\in (pV)_*$ satisfying
$$\norm{x}_{\varphi_1,\varphi_2}\le   \sqrt2\cdot \sqrt{\norm{\varphi_1}+\norm{\varphi_2}}\cdot \norm{x}_\psi,\mbox{ for all } x\in pV.$$
\end{prop}

\begin{proof}
For $j=1,2$ let $\psi_j$ be a positive functional on $pVp$ and $u_j\in V$ a unitary element provided by Lemma~\ref{L:shift phi to state} for $\varphi_j$.
Set
$$t=p\vee u_1pu_1^*\vee u_2pu_2^*$$
and $W=tVt$. Then $t$, being the supremum of three projections equivalent to $p$, is a finite projection (cf. \cite[Theorem V.1.37]{Tak}). Moreover, the central carrier (also called the central support) of $p$ in $W$ equals $t=1_W$ (just observe that if $z$ is a central projection in $W$ with $z p =0,$ then $z u_j p u_j^* = z u_j p u_j^* z =0$ for all $j=1,2$, and hence $z=0$).\smallskip

{
We claim that $W$ is of type $II_1$. Indeed, assume that $r\in W$ is a nonzero abelian projection. Since the central carrier of $p$ equals $1_W$ \cite[Lemma V.1.25]{Tak} yields a nonzero
 projections $r_1\le r$ such that $r_1\sim p$.  Since $r_1$ is abelian, $p$ is abelian, too, which contradicts the assumption that $pVp$ is of type $II_1$.}\smallskip

Moreover, for $j=1,2$ we have $s(\varphi_j)=s(\psi_j)u_j^*$, so the inital projection is
$u_js(\psi_j)u_j^*\le u_jpu_j^*\le t$, hence $s(\varphi_j)\in pVt=pW$. By Lemma~\ref{L:max of double norm} the pre-Hilbert semi-norm $\norm{\cdot}_{\varphi_1,\varphi_2}$ attains its maximum on the closed unit ball of $pVt$. We deduce from Lemma \ref{l 5.5 new} that $\norm{\cdot}_{\varphi_1,\varphi_2}$ actually attains its maximum on the closed unit ball of $pV$. Thus, by Lemma \ref{l:double NA -> estimate}, there is a norm-one functional $\psi\in (pV)_*$ such that
$$\norm{x}_{\varphi_1,\varphi_2}\le \sqrt{2}\cdot\sqrt{\norm{\varphi_1}+\norm{\varphi_2}} \cdot \norm{x}_\psi, \qquad x\in pV.$$
\end{proof}

So, we have solved the case of the summand $p_2V$ from Proposition~\ref{P:representation} and we turn our attention to the remaining summand $p_3V$.
\smallskip

Henceforth, for each natural $n$, the symbol $M_n$ will stand for the C$^*$-algebra of all $n\times n$-matrices with complex entries. Given $1\le k\le n$, we shall denote by $\U(M_n)$ the set of all unitary matrices in $M_n$, and by $P_k(M_n)$ the set of all projections of rank $k$.

\begin{lemma}\label{L:projections in Mn} The following assertions hold:
\begin{enumerate}[$(a)$]
\item Any two projections $q_1,q_2\in P_k(M_n)$  are unitarily equivalent;
\item $P_k(M_n)$ is a compact set;
\item given $r\in\P_k(M_n)$ there is a Borel measurable function $\upsilon:P_k(M_n)\to \U(M_n)$ such that
$$r=\upsilon(q)^* q \upsilon(q)\mbox{ for all }q\in P_k(M_n).$$
\end{enumerate}
\end{lemma}

\begin{proof}$(a)$ This is well known and easy to see.\smallskip

$(b)$ It is clear that $\U(M_n)$ is a compact set and that the mapping
$$u\mapsto uru^*,\quad u\in \U(M_n),$$
where $r\in P_k(M_n)$ is fixed, is a continuous map of $\U(M_n)$ onto $P_k(M_n)$. Thus, $P_k(M_n)$ is compact.\smallskip

$(c)$ Fix $r\in P_k(M_n)$ and consider the continuous mapping used in $(b)$. The inverse of this mapping admits a Borel measurable selection by the Kuratowski-Ryll-Nardzewski theorem (cf. \cite[Theorem 18.13]{aliprantisborder}). Denote the selection by $\upsilon$. Then
$$\upsilon(q) r \upsilon(q)^*=q\mbox{ for all }q\in P_k(M_n),$$
hence the assertion follows.
\end{proof}

\begin{lemma}\label{L:projections in LinftyMn}
Let $W=L^\infty(\mu,M_n)$ for a probability measure $\mu$ and $n\in\en$.
\begin{enumerate}[$(a)$]
\item An element $f\in W$ is a projection if and only if $f(\omega)$ is a projection in $M_n$ for $\mu$-almost all $\omega$;
\item Any projection $f\in W$ is unitarily equivalent to a projection $g\in W$ such that $g(\omega)\in\{0,r_1,\dots,r_{n-1},I\}$ for $\mu$-almost all $\omega$, where $r_j\in M_n$ is a fixed projection of rank $j$ for $1\le j<n$.
\end{enumerate}
\end{lemma}

\begin{proof}$(a)$ This assertion follows immediately from definitions.\smallskip

$(b)$ Let $f\in W$ be a projection. For $k\in\{0,\dots,n\}$ let
$$A_k=\{\omega\setsep \dim\ran f(\omega)=k\}.$$
By Lemma~\ref{L:projections in Mn}$(b)$ each $A_k$ is $\mu$-measurable, being a preimage of a compact set. Further, for each $k\in\{1,\dots,n\}$ let $\upsilon_k:P_k(M_n)\to\U(M_n)$ be the mapping provided by Lemma~\ref{L:projections in Mn}$(c)$ for the projection $r_k$. Set
$$u(\omega)=\begin{cases} I & \omega\in A_0\cup A_n,\\ \upsilon_k(\omega) & \omega\in A_k, 0<k<n.
\end{cases}$$
Then $u$ is a unitary element of $W$ and $g=u^*fu$ is a projection satisfying the required properties.
\end{proof}

\begin{lemma}\label{L:representation pLinftyMn}
Let $W=L^\infty(\mu,M_n)$ for a probability measure $\mu$ and $n\in\en$. Let $p\in W$ be a projection.
Then the JB$^*$-triple $pW$ is JB$^*$-triple isomorphic to
$$\bigoplus_{1\le k\le n}^{\ell_\infty} L^\infty(\mu_k, r_kM_n),$$
where $\mu_k$ is a finite non-negative measure and $r_k\in M_n$ is a projection of rank $k$ for each $k\in\{1,\dots,n\}$.
\end{lemma}

\begin{proof} For each $k\in \{0,\dots,n\}$ let $r_k\in M_n$ be a projection of rank $k$ (note that $r_0=0$ and $r_n=I$).
By Lemma~\ref{L:projections in LinftyMn} $p$ is unitarily
equivalent to a projection $g$ such that $g(\omega)\in\{r_0,\dots,r_n\}$ $\mu$-almost everywhere.
Then $pW$ is triple-isomorphic to $gW$. Further, for $k=0,\dots,n$ set
$$A_k=\{\omega\setsep g(\omega)=r_k\}.$$
Then
$$gW=\bigoplus_{1\le k\le n}^{\ell_\infty} L^\infty(\mu|_{A_k},r_kM_n),$$
which completes the proof.
\end{proof}

\begin{lemma}\label{L:representation type I}
Let $V$ be a finite von Neumann algebra of type $I$ and let $p\in V$ be a projection. Then the JB$^*$-triple $pV$ is JB$^*$-triple isomorphic to
$$\bigoplus_{j\in J}^{\ell_\infty} L^\infty(\mu_j,p_j M_{n_j}),$$
where $\mu_j$ is a probability measure, $n_j\in\en$ and $p_j\in M_{n_j}$ is a projection for $j\in J$.
\end{lemma}

\begin{proof}
By combining \cite[Theorem V.1.27]{Tak} and \cite[Corollary V.2.9]{Tak} we get an orthogonal family $(z_\alpha)_{\alpha\in\Lambda}$ of central projections in $V$ with sum equal to $1$ such that
$z_\alpha V$ is isomorphic to $A_\alpha\overline{\otimes} M_{n_\alpha}$, where $A_\alpha$ is a $\sigma$-finite abelian von Neumann algebra and $n_\alpha\in\en$ for $\alpha\in\Lambda$. Each $A_\alpha$, being $\sigma$-finite, is isomorphic to $L^\infty(\mu_\alpha)$ for some probability measure $\mu_\alpha$. Thus $\displaystyle pV=\bigoplus_{\alpha\in \Lambda}pz_\alpha V$ is isomorphic to
$$\bigoplus_{\alpha\in\Lambda}^{\ell_\infty} z_\alpha p L^\infty(\mu_\alpha,{M}_{n_\alpha}).$$
We conclude by applying Lemma~\ref{L:representation pLinftyMn} to each summand.
\end{proof}

The following proposition solves the case of the summand $p_3V$ from Proposition~\ref{P:representation}.

\begin{prop}\label{p pV with p finite and type I} Let $V$ be a von Neumann algebra and $p\in V$ a finite projection such that $pVp$ is of type $I$.
Then for any normal functionals $\varphi_1,\varphi_2\in (pV)_*$ { the semi-norm $\norm{\cdot}_{\varphi_1,\varphi_2}$ attains its maximum on the unit ball of $pV$ and therefore} there exists a norm-one functional $\psi\in (pV)_*$ satisfying
$$\norm{x}_{\varphi_1,\varphi_2}\le {\sqrt{2}\cdot\sqrt{\norm{\varphi_1}+\norm{\varphi_2}}}\cdot \norm{x}_\psi\mbox{ for }x\in pV.$$
\end{prop}

\begin{proof}
For $j=1,2$ let $\psi_j$ be a positive functional on $pVp$ and $u_j\in V$ a unitary element provided by Lemma~\ref{L:shift phi to state} for $\varphi_j$.
Set
$$t=p\vee u_1pu_1^*\vee u_2pu_2^*$$
and $W=tVt$. Then $t$, being the supremum of three projections equivalent to $p$, is a finite projection. Moreover, the central carrier of $p$ in $W$ equals $t=1_W$.\smallskip

We claim that $W$ is of type $I$.
Indeed, assume that $r\in W$ is a nonzero  projection. Since the central carrier of $p$ equals $1_W$, \cite[Lemma V.1.7]{Tak} yields that there are two nonzero projections $r_1\le r$ and $p_1\le p$ such that $r_1\sim p_1$. Since $pVp$ is of type $I$, there is a nonzero abelian projection $p_2\le p_1$. Then there is a projection $r_2\le r_1$ equivalent to $p_2$. Therefore $r_2$ is abelian and $r_2\le r_1\le r$, which completes the proof of the claim.\smallskip

Moreover, for $j=1,2$ we have $s(\varphi_j)=s(\psi_j)u_j^*$, so the inital projection is
$u_js(\psi_j)u_j^*\le u_jpu_j^*\le t$, hence $s(\varphi_j)\in pVt=pW$. {\ By Lemma~\ref{L:representation type I} $pW = p W t $ is JB$^*$-triple isomorphic to
$\displaystyle \bigoplus_{j\in J}^{\ell_\infty} L^\infty(\mu_j,p_j M_{n_j}),$
where $\mu_j$ is a probability measure, $n_j\in\en$ and $p_j\in M_{n_j}$ is a projection for $j\in J$. For each $j\in J$, let $\varphi_{1,j} = \varphi_1|_{L^\infty(\mu_j,p_j M_{n_j})}$ and $\varphi_{2,j} = \varphi_2|_{L^\infty(\mu_j,p_j M_{n_j})}$. Proposition \ref{p LG Linfty mu C} assures that the pre-Hilbert semi-norm $\|.\|_{\varphi_1,\varphi_2}|_{L^\infty(\mu_j,p_j M_{n_j})} = \|.\|_{\varphi_{1,j},\varphi_{2,j}}$ attains its maximum on the closed unit ball of $L^\infty(\mu_j,p_j M_{n_j})$
{at some point $x_j$. It follows
  that the semi-norm $\|.\|_{\varphi_1,\varphi_2}$ attains its maximum on the closed unit ball of $p W = p V t$ at the point $(x_j)_{j\in J}$.} We can therefore apply Lemma \ref{l 5.5 new} to deduce that $\|.\|_{\varphi_1,\varphi_2}$ attains its maximum on the closed unit ball of $p V $. Finally, Lemma \ref{l:double NA -> estimate} yields a norm-one functional $\psi\in (pV)_*$
such that
$$\norm{x}_{\varphi_1,\varphi_2}\le \sqrt{2} \sqrt{\norm{\varphi_1}+\norm{\varphi_2}}\, \norm{x}_\psi, \hbox{ for all } x\in pV.$$}
\end{proof}

\section{Proof of Grothendieck's inequalities for JB$^*$-triples} 

Now we are ready to prove the Barton-Friedman conjecture. We start by restating and proving the little Grothendieck inequality given in Theorem~\ref{ttt solution to LG problem JBW*-triples}.

{
\begin{thm}\label{tt solution to LG problem JBW*-triples} Let $M$ be a JBW$^*$-triple. Then given any two functionals $\varphi_1,\varphi_2$ in $M_*$, there exists a norm-one functional $\psi\in M_*$ such that
$$\norm{x}_{\varphi_1,\varphi_2}\le {\sqrt2\cdot\sqrt{\norm{\varphi_1}+\norm{\varphi_2}}}\cdot \norm{x}_\psi,$$ for all $x\in M.$ Furthermore, given $K> { 2}$, for every complex Hilbert space $H$, and every weak$^*$-to-weak continuous linear operator $T:M \rightarrow {H}$, there exists a norm-one functional $\psi\in M_{*}$ satisfying  $$\|T(x)\| \leq K \, \|T\| \, \|x\|_{\psi}$$  for all $x\in
M$.
\end{thm}

\begin{proof} The first statement follows from the results of the previous section. Indeed,
consider the decomposition of $M$ from Proposition~\ref{P:representation}. The statement for individual summands follows from  Proposition \ref{p LG Linfty mu C},
Corollary \ref{c LG when Peirce-2 are upward directed}, Proposition \ref{p pV p finite type II1}, and Proposition \ref{p pV with p finite and type I}, respectively. Finally, Proposition~\ref{p stability of the LG under ellinfty sums precise} completes the argument.
\smallskip

Let us prove the second statement. Fix $K>2$. Let $\varepsilon>0$ be such that $K>2(1+\varepsilon)$. By Theorem~\ref{t little complex w*} there are norm-one functionals $\varphi_1,\varphi_2\in M_*$ such that for any $x\in M$ we have
$$\norm{T(x)}\le \sqrt{K}\norm{T}\sqrt{\norm{x_1}_{\varphi_1}^2+\varepsilon\norm{x}_{\varphi_2}^2}=\sqrt{K}\norm{T}\norm{x}_{\varphi_1,\varepsilon\varphi_2}.$$
By the first part of the theorem we get a norm-one functional $\psi\in M_*$ such that for $x\in M$ we have
$$\norm{x}_{\varphi_1,\varepsilon\varphi_2}\le \sqrt{2}\sqrt{\norm{\varphi_1}+\norm{\varepsilon\varphi_2}}\norm{x}_\psi=\sqrt{2(1+\varepsilon)}\norm{x}_\psi.$$
By combining the two inequalities we get
$$\norm{T(x)}\le \sqrt{2(1+\varepsilon)K} \norm{x}_\psi\le K\norm{x}_\psi$$
for $x\in M$.
This completes the proof.\end{proof}
}

Given a bounded linear operator $T$ from a JB$^*$-triple $E$ into a complex Hilbert space $H$ we can always consider its bitranspose $T^{**}: E^{**}\to H$, which is a weak$^*$-to-weak continuous linear operator from a JBW$^*$-triple into a complex Hilbert space. We therefore arrive, via Theorem \ref{tt solution to LG problem JBW*-triples}, to a proof of the little Grothendieck inequality with one control functional.

\begin{thm}\label{t solution to LG problem JB*-triples} Let $E$ be a JB$^*$-triple, $H$ a complex Hilbert space, and $K>{ 2}$. Then for every bounded linear operator $T:E \rightarrow {H}$, there exists a norm-one functional $\psi\in E^*$ satisfying  $$\|T(x)\| \leq K \, \|T\| \, \|x\|_{\psi}$$  for all $x\in E$.$\hfill\Box$
\end{thm}

The previous Theorems \ref{tt solution to LG problem JBW*-triples} and \ref{t solution to LG problem JB*-triples} restore the equilibrium and the validity of original statements concerning the little Grothendieck inequality in the case of JB$^*$-triples in \cite{barton1987grothendieck,chu1989grothendieck}. It also provides a complete solution to \cite[Problem 5.10.131]{Cabrera-Rodriguez-vol2}, \cite[Remark 3]{peralta2001grothendieck}, and \cite[Remark 8.3]{pisier2012grothendieck}. We shall next trace back the original sources to see how our results can be also employed to provide a complete proof to the Barton--Friedman conjecture concerning Grothendieck's inequality for bilinear forms on JB$^*$-triples.

\begin{thm}\label{t big Grothendieck inequ JBW} Suppose $G > {8 (1+2\sqrt{3})}$. Let $M$ and $N$ be JBW$^*$-triples. Then for every separately
weak$^*$-continuous bilinear form $V:M\times N\to \mathbb{C}$ there exist norm-one functionals $\varphi\in M_{*}$ and $\psi\in N_{*}$ satisfying
$$|V(x,y)| \leq G \ \|V\| \, \|x\|_{\varphi} \, \|y\|_{\psi}$$ for all $(x,y)\in M \times N$.
\end{thm}

\begin{proof} Thanks to our previous Theorem \ref{tt solution to LG problem JBW*-triples} we can recover a trick from \cite[Theorem 6]{chu1989grothendieck} and \cite[Remark 3]{peralta2001grothendieck}. A brief argument is included here for completeness reasons. Let us find a weak$^*$-to-weak continuous linear operator $R: M\to N_*$ defined by $V(a,b) = \langle R(a) , b\rangle$ ($(a,b)\in M\times N$). Clearly $\|R\|\leq \|V\|$. By \cite[Lemma 5]{chu1989grothendieck} $R$ factors through a complex Hilbert space, more precisely, there exists a complex Hilbert
space $H$ and bounded linear operators $T: M \rightarrow H$, $S: H \rightarrow N_{*}$ satisfying $R = S \circ T$ and
$\|T\|\ \|S\| \leq 2(1+2\sqrt{3})\ \|R\|$. It is further shown in the proof of \cite[Theorem 6]{peralta2001grothendieck} that we can choose $H$ in such a way that $S$ is injective and $T$ is weak$^*$-to-weak continuous.\smallskip

Let $\displaystyle \tilde{G} = \left(\frac{G}{2 (1+2\sqrt{3})} \right)^{\frac12} >  2.$ By applying Theorem \ref{tt solution to LG problem JBW*-triples} to the weak$^*$-to-weak continuous linear operators $T: M\to H$ and $S^* : N\to H$ we find two norm-one functionals $\varphi\in M_{*}$ and $\psi\in N_{*}$ satisfying $$\|T(x)\| \leq \tilde{G} \, \|T\| \, \|x\|_{\varphi}, \hbox{ and } \|S^*(y)\| \leq \tilde{G} \, \|S^*\| \, \|y\|_{\psi}$$  for all $(x,y)\in M\times N$. We therefore have

$$|V(x,y)|=|\langle R(x),y \rangle|=|\langle T(x),S^{*} (y)\rangle | \leq \tilde{G}^2\, \|T\| \, \|S\| \, \|x\|_{\varphi}\,
\|y\|_{\psi} $$ $$\leq G \, \|V\| \, \|x\|_{\varphi}\, \|y\|_{\psi} $$ for all $(x,y)\in M \times N$.
\end{proof}

Since every bounded bilinear form on the cartesian product of two JB$^*$-triples admits a norm-preserving separately weak$^*$-continuous extension to the cartesian product of the corresponding bidual spaces (cf. \cite[Lemma 1]{peralta2001grothendieck}), Theorem \ref{t big Grothendieck inequ JBW} implies the following statement (restating of Theorem~\ref{tt big Grothendieck inequ} from Introduction).

\begin{thm}\label{t big Grothendieck inequ} Suppose $G >{ 8 (1+2\sqrt{3})}$. Let $E$ and $B$ be JB$^*$-triples. Then for every bounded 
bilinear form $V:E\times B\to \mathbb{C}$ there exist norm-one functionals $\varphi\in E^{*}$ and $\psi\in B^{*}$ satisfying
$$|V(x,y)| \leq G \ \|V\| \, \|x\|_{\varphi} \, \|y\|_{\psi}$$ for all $(x,y)\in E \times B$.$\hfill\Box$
\end{thm}

\begin{remark} 
The optimal values of the constants in question remain to be unknown. However, it seems that our method cannot give a better constant in Theorem~\ref{tt solution to LG problem JBW*-triples}. 
One factor $\sqrt2$ appears due to the use of Lemma~\ref{lemma little w* norm attaining} and a second factor $\sqrt{2}$ appears due to estimates of semi-norms $\norm{\cdot}_{\varphi_1,\varphi_2}$ by a semi-norm generated by one functional. Let us consider a JBW$^*$-triple represented as in Proposition~\ref{P:representation}. The individual summands have different behaviour.
\begin{enumerate}[$(i)$]
    \item The JBW$^*$-algebra $N$ is covered by the already known Theorem~\ref{t LG JB*algebras}.
    \item The summand $p_1V$ is covered by Corollary~\ref{c LG when Peirce-2 are upward directed}. This approach can be applied to $N$ as well (note that Corollary~\ref{c LG when Peirce-2 are upward directed} can be viewed as a generalization of Theorem~\ref{t LG JB*algebras}).
    \item The remaining summand, i.e.,
    \begin{equation*} \left( \bigoplus_{k\in\Lambda}^{\ell_{\infty}} L^\infty(\mu_k,C_k)\right)\oplus^{\ell_{\infty}} p_2V \oplus^{\ell_{\infty}} p_3V,
\end{equation*} 
    has a special property. It follows from our arguments that in this case $\norm{\cdot}_{\varphi_1,\varphi_2}$ attains its maximum on the unit ball for any two normal functionals $\varphi_1,\varphi_2$.
\end{enumerate}
This analysis confirms that there are two basic tools -- attaining the norm and some kind of order on tripotents.
\end{remark}

\begin{remark}
Recently, the constants in the Little Grothendieck Theorem for JB$^*$-algebras have been improved in \cite{kalenda2020optimality}.  In particular, by \cite[Theorem 6.3]{kalenda2020optimality} the constant $2$ in Theorem~\ref{t LG JB*algebras} may be replaced by $\sqrt{2}$. Further, if $E$ is a JB$^*$-algebra, then the constant $2$ in Theorem~\ref{t solution to LG problem JB*-triples}  may be replaced by $\sqrt{2}$ due to \cite[Theorem 1.3]{kalenda2020optimality}. 

However, it is not clear whether one can take in these cases even $K=\sqrt{2}$ or whether a similar improvement holds for general JB$^*$-triples. 
\end{remark}

\def\cprime{$'$} \def\cprime{$'$}


\begin{thebibliography}{10}

\bibitem{aliprantisborder}
{\sc Aliprantis, C.~D., and Border, K.~C.}
\newblock {I}nfinite dimensional analysis: A hitchhikers guide. 3rd ed., 2006.

\bibitem{barton1987grothendieck}
{\sc Barton, T., and Friedman, Y.}
\newblock Grothendieck's inequality for {JB}$^*$-triples and applications.
\newblock {\em Journal of the London Mathematical Society 2}, 3 (1987),
  513--523.

\bibitem{BaTi}
{\sc Barton, T., and Timoney, R.~M.}
\newblock Weak{$^\ast$}-continuity of {J}ordan triple products and its
  applications.
\newblock {\em Math. Scand. 59}, 2 (1986), 177--191.

\bibitem{braun1978holomorphic}
{\sc Braun, R., Kaup, W., and Upmeier, H.}
\newblock A holomorphic characterization of {J}ordan {C}*-algebras.
\newblock {\em Mathematische Zeitschrift 161}, 3 (1978), 277--290.

\bibitem{BuFPMaMoPe}
{\sc Bunce, L.~J., Fern\'{a}ndez-Polo, F.~J., Mart\'{i}nez~Moreno, J., and
  Peralta, A.~M.}
\newblock A {S}ait\^{o}-{T}omita-{L}usin theorem for {JB}*-triples and
  applications.
\newblock {\em Q. J. Math. 57}, 1 (2006), 37--48.

\bibitem{Cabrera-Rodriguez-vol1}
{\sc Cabrera~Garc\'{\i}a, M., and Rodr\'{\i}guez~Palacios, A.}
\newblock {\em Non-associative normed algebras. {V}ol. 1}, vol.~154 of {\em
  Encyclopedia of Mathematics and its Applications}.
\newblock Cambridge University Press, Cambridge, 2014.
\newblock The Vidav-Palmer and Gelfand-Naimark theorems.

\bibitem{Cabrera-Rodriguez-vol2}
{\sc Cabrera~Garc\'{i}a, M., and Rodr\'{i}guez~Palacios, A.}
\newblock {\em Non-associative normed algebras. {V}ol. 2}, vol.~167 of {\em
  Encyclopedia of Mathematics and its Applications}.
\newblock Cambridge University Press, Cambridge, 2018.
\newblock Representation theory and the Zel'manov approach.

\bibitem{chubook}
{\sc Chu, C.-H.}
\newblock {\em Jordan structures in geometry and analysis}, vol.~190 of {\em
  Cambridge Tracts in Mathematics}.
\newblock Cambridge University Press, Cambridge, 2012.

\bibitem{chu1989grothendieck}
{\sc Chu, C.-H., Iochum, B., and Loupias, G.}
\newblock Grothendieck's theorem and factorization of operators in {J}ordan
  triples.
\newblock {\em Mathematische Annalen 284}, 1 (1989), 41--53.

\bibitem{Dineen1986seconddual}
{\sc Dineen, S.}
\newblock Complete holomorphic vector fields on the second dual of a {B}anach
  space.
\newblock {\em Math. Scand. 59}, 1 (1986), 131--142.

\bibitem{EdMaFPHoPe}
{\sc Edwards, C.~M., Fern\'{a}ndez-Polo, F.~J., Hoskin, C.~S., and Peralta,
  A.~M.}
\newblock On the facial structure of the unit ball in a {$\rm JB^*$}-triple.
\newblock {\em J. Reine Angew. Math. 641\/} (2010), 123--144.

\bibitem{fabian-book}
{\sc Fabian, M.~J.}
\newblock {\em G\^{a}teaux differentiability of convex functions and topology}.
\newblock Canadian Mathematical Society Series of Monographs and Advanced
  Texts. John Wiley \& Sons, Inc., New York, 1997.
\newblock Weak Asplund spaces, A Wiley-Interscience Publication.

\bibitem{Friedman-Russo}
{\sc Friedman, Y., and Russo, B.}
\newblock Structure of the predual of a {$JBW^\ast$}-triple.
\newblock {\em J. Reine Angew. Math. 356\/} (1985), 67--89.

\bibitem{grothendieck1956resume}
{\sc Grothendieck, A.}
\newblock {\em R{\'e}sum{\'e} de la th{\'e}orie m{\'e}trique des produits
  tensoriels topologiques}.
\newblock Soc. de Matem{\'a}tica de S{\~a}o Paulo, 1956.

\bibitem{haagerup1985grothendieck}
{\sc Haagerup, U.}
\newblock The {G}rothendieck inequality for bilinear forms on {C}$^*$-algebras.
\newblock {\em Advances in Mathematics 56}, 2 (1985), 93--116.

\bibitem{hamhalter2019mwnc}
{\sc Hamhalter, J., Kalenda, O. F.~K., Peralta, A.~M., and Pfitzner, H.}
\newblock Measures of weak non-compactness in preduals of von {N}eumann
  algebras and {$\rm JBW^\ast$}-triples.
\newblock {\em J. Funct. Anal. 278}, 1 (2020), 108300.

\bibitem{hanche1984jordan}
{\sc Hanche-Olsen, H., and St{\o}rmer, E.}
\newblock {\em Jordan operator algebras}, vol.~21.
\newblock Pitman Advanced Publishing Program, 1984.

\bibitem{harris1974bounded}
{\sc Harris, L.~A.}
\newblock Bounded symmetric homogeneous domains in infinite dimensional spaces.
\newblock In {\em Proceedings on infinite dimensional holomorphy\/} (1974),
  Springer, pp.~13--40.

\bibitem{horn1987classification}
{\sc Horn, G.}
\newblock Classification of {JBW}*-triples of type {I}.
\newblock {\em Mathematische Zeitschrift 196}, 2 (1987), 271--291.

\bibitem{horn1988classification}
{\sc Horn, G., and Neher, E.}
\newblock Classification of continuous {JBW}*-triples.
\newblock {\em Transactions of the American Mathematical Society 306}, 2
  (1988), 553--578.

\bibitem{KR2}
{\sc Kadison, R.~V., and Ringrose, J.~R.}
\newblock {\em Fundamentals of the theory of operator algebras. {V}ol. {II}},
  vol.~16 of {\em Graduate Studies in Mathematics}.
\newblock American Mathematical Society, Providence, RI, 1997.
\newblock Advanced theory, Corrected reprint of the 1986 original.

\bibitem{kalenda2020optimality}
{\sc Kalenda, O. F.~K., Peralta, A.~M., and Pfitzner, H.}
\newblock On optimality of constants in the Little Grothendieck theorem, arXiv:2002.12273.

\bibitem{kaup1983riemann}
{\sc Kaup, W.}
\newblock A {R}iemann mapping theorem for bounded symmetric domains in complex
  {B}anach spaces.
\newblock {\em Mathematische Zeitschrift 183}, 4 (1983), 503--529.

\bibitem{kaup1977jordan}
{\sc Kaup, W., and Upmeier, H.}
\newblock Jordan algebras and symmetric {S}iegel domains in {B}anach spaces.
\newblock {\em Mathematische Zeitschrift 157}, 2 (1977), 179--200.

\bibitem{peralta2001little}
{\sc Peralta, A.~M.}
\newblock Little {G}rothendiecks theorem for real {JB$^*$}-triples.
\newblock {\em Mathematische Zeitschrift 237}, 3 (2001), 531--545.

\bibitem{peralta2005new}
{\sc Peralta, A.~M.}
\newblock New advances on the {G}rothendieck's inequality problem for bilinear
  forms on {JB$^*$}-triples.
\newblock {\em Math. Inequal. Appl 8\/} (2005), 7--21.

\bibitem{peralta2001grothendieck}
{\sc Peralta, A.~M., and Rodr{\'\i}guez-Palacios, A.}
\newblock Grothendieck's inequalities for real and complex {$\rm
  JBW^*$}-triples.
\newblock {\em Proceedings of the London Mathematical Society 83}, 3 (2001),
  605--625.

\bibitem{peralta2001grothendieckrev}
{\sc Peralta, A.~M., and Rodr{\'\i}guez-Palacios, A.}
\newblock Grothendieck’s inequalities revisited.
\newblock In {\em North-Holland Mathematics Studies}, vol.~189. Elsevier, 2001,
  pp.~409--423.

\bibitem{pisier1978grothendieck}
{\sc Pisier, G.}
\newblock Grothendieck's theorem for noncommutative {C}$^*$-algebras, with an
  appendix on {G}rothendieck's constants.
\newblock {\em Journal of Functional Analysis 29}, 3 (1978), 397--415.

\bibitem{pisier2012grothendieck}
{\sc Pisier, G.}
\newblock Grothendieck’s theorem, past and present.
\newblock {\em Bulletin of the American Mathematical Society 49}, 2 (2012),
  237--323.

\bibitem{Tak}
{\sc Takesaki, M.}
\newblock {\em Theory of operator algebras. {I}}.
\newblock Springer-Verlag, New York-Heidelberg, 1979.

\bibitem{Wright1977}
{\sc Wright, J. D.~M.}
\newblock Jordan {$C\sp*$}-algebras.
\newblock {\em Michigan Math. J. 24}, 3 (1977), 291--302.

\end{thebibliography}
\end{document}